\documentclass[11pt]{amsart} 

\usepackage{stile2017}
\usepackage{algorithm,algpseudocode} 
\usepackage{autonum}		
\usepackage{tikz}

\declarebtxcommands{english}{%
}

\newcommand{\widebar}{\overline}

\title[On the unfair 0-1-polynomials conjecture]{Progress on the unfair 0-1-polynomials conjecture using linear recurrences and numerical analysis}
\author{Luca Ghidelli}
\address{Luca Ghidelli, Bunsenstr. 3-5, Mathematisches Institut, D-37073 G{\"o}ttingen Germany}
\email{{lghidel@mathematik.uni-goettingen.de}}
\date{\today}



\renewcommand{\Re}{\op{Re}}

\renewcommand{\Im}{\op{Im}}

\newcommand{\pp}[1]{\tilde P (#1)}
\newcommand{\ppt} [2]{\tilde P_{#1} (#2)}

\newcommand{\pptx} [1]{x^5+{#1}x^3+1}

\newcommand{\aand}{\quad\text{and}\quad}

\begin{document}
	
\maketitle

\begin{abstract}
	If the product of two monic polynomials with real nonnegative coefficients has all coefficients equal to 0 or 1, does it follow that all the coefficients of the two factors are also equal to 0 or 1? 
	Here is an equivalent formulation of this intriguing problem: is it possible to weigh unfairly a pair of dice so that the probabilities of every possible outcome (roll them and take the sum) were the same? If the two dice have six faces numbered 1 to 6, it is easy to show that the answer is no. But for general dice with finitely many faces, this is an open problem with no significant advancement since 1937. In this paper we examine, in some sense, the first infinite family of cases that cannot be treated with classical methods: the first die has three faces numbered with 0,2 and 5, while the second die is arbitrary. In other words, we examine factorizations of 0-1-polynomials with one factor equal to $x^5+a x^2 +1$ for some nonnegative $a$. We discover that this case may be solved (that is, necessarily $a=0$ or $a=1$) using the theory of linear recurrence sequences, computation of resultants, a fair amount of analytic and numerical approximations... and a little bit of luck. 
\end{abstract}
\tableofcontents

\section{Introduction}

The objective of this article is to establish some partial progress on the following unsolved problem in probability:
\begin{conjecture}[main conjecture, probabilistic version]\label{conj:prob}
Let $X,Y$ be discrete independent random variables with finite support in the natural numbers and suppose that the random variable $Z=X+Y$ is uniform on its support. Then $X,Y$ are uniform on their respective supports as well. 
\end{conjecture}

We may think of $X$ and $Y$ as describing the outcomes of rolling a pair of dice, whose finitely many faces are labelled with natural numbers. Then \cref{conj:prob} may be stated in the following appealing fashion: \emph{Is it possible to weigh a pair of unfair dice so that the probabilities of every possible sum of their outcome were the same?}

When $X,Y,Z$ are random variables related by $Z=X+Y$ we also say that $X$ and $Y$ are \emph{factors} of $Z$. 
This ``multiplicative'' terminology is motivated by the fact that in this situation, on the dual side, the characteristic functions $\phi_V(\tau):=\mbb E[e^{i\tau V}]$ for $V\in\{X,Y,Z\}$ are related by $\phi_Z(\tau)=\phi_Y(\tau)\phi_X(\tau)$. 
If $V$ is a discrete random variable taking values on finitely many natural numbers, it is more natural to change variable $t=e^{i\tau}$ and consider the probabiliy generating function $G_V(t) := \mbb E [t^V]$, which is in fact a polynomial in the variable $t$. 
If we normalize the probability generating polynomials of $X,Y,Z$ in \cref{conj:prob} above, dividing through by their leading coefficients, we obtain the following equivalent algebraic statement:

\begin{conjecture}[main conjecture, algebraic version]\label{conj:alg}
Let $P(t),Q(t)$ be monic polynomials with nonnegative real coefficients and suppose that  $R(t)=P(t)Q(t)$ has all of its coefficients equal to either 0 or 1. Then all coefficients of $P$ and $Q$ are in $\{0,1\}$ as well. 
\end{conjecture}

Since both the equivalent statements displayed in \cref{conj:prob} and \cref{conj:alg} are very natural, one is tempted to refer to them jointly as a \emph{conjecture on nonnegative factorization of 0-1-polynomials and unfair dice with fair sum}. Because of the length of such resulting expression, I propose a convenient hybrid naming system.

\begin{definition}\label{def:hybrid}
	Let $R(t)$ be a 0-1-polynomial, that is, a polynomial with all coefficients equal to 0 or 1. We say a factorization $R(t)=P(t)Q(t)$ is \emph{fair} if both $P$ and $Q$ are 0-1-polynomials. We say the factorization is \emph{unfair} if both $P$ and $Q$ are monic and have real nonnegative coefficient, but they are not (both) 0-1-polynomials. A 0-1-polynomial $R(t)$ is \emph{unfair} if it admits an unfair factorization, and it is called \emph{fair} otherwise.
\end{definition}

Then, we will refer to \cref{conj:prob,conj:alg} jointly as \emph{the unfair 0-1-polynomial conjecture}. With the notation introduced in \cref{def:hybrid}, we may express the conjecture e.g. in the following form:
\begin{conjecture}[main conjecture, using \protect\cref{def:hybrid}]\label{conj:hybrid}
	There are no unfair factorizations of 0-1-polynomials.
\end{conjecture}

The study of factors of random variables is a classical problem in the Arithmetic of Probability Distributions \cite{arithmetic:laws}. This is a branch of Mathematics, also known as Algebraic Probability Theory \cite{ruzsa1988}, that studies the semigroup of probability distributions under convolution. 
Several classical theorems show that some classes of random variables are ``closed under the operation of taking factors''. This is the case, for instance, for the random variables with Gaussian distribution \cite{cramer1936}, or those with Poissonian distribution \cite{raikov1937,raikov1938}. If true, the unfair 0-1-polynomials conjecture would imply that discrete and uniform random variables with finite support in the integers have this property as well. 
Since we deal with uniform distributions, it is worth mentioning a celebrated result of Lewis \cite{lewis1967}, who described the factors of the uniform distribution on the interval $[0,1]$.

The polynomials that have all their coefficients equal to 0 or 1 are known in the literature as Newman polynomials or sometimes simply as 0-1-polynomials. Thus \cref{conj:alg} speculates about an interesting conjectural factorization property of Newman polynomials. This problem has a relatively long history. In 1937 Krasner and Ranulac established the unfair 0-1-polynomials conjecture hold for polynomials of the form $R(t)=(t^{N}-1)/(t-1)$ with $N\in\N$, see \cite{krasner1937,raikov1937b,behrends1999,ruzsa1988,lepetitThesis}. This is the same as proving \cref{conj:prob} for random variables $Z$ uniformly supported on sets of of the form $\{0,1,2,\ldots,N-1\}$. Their argument is elementary and  it applies more generally to polynomials that are ``palindromic'', i.e. $R(t)=t^{\deg R} R(t^{-1})$.

Krasner and Ranulac's result, or its generalization to the palindromic case, has appeared independently also in the dice version, e.g. in \cite{kelly1951e925,morrison2018,huang2019}. The special case with two cubic dice with six faces numbered 1 to 6 is a well-known problem in recreational mathematics. It admits an even simpler solution, which we leave to the reader to discover. Hint: it corresponds to a case of \cref{conj:alg} with $\deg P = \deg Q$.

To the best of the author's investigations, \cref{conj:prob} was asked for the first time in the stated general form by G. Letac around 1969. 
Since then, except for the mentioned results on palindromic polynomials, there has been essentially no progress on this conjecture. 
The conjecture regained some popularity a few years ago, when E. Amiot posted this as a question on MathStackExchange. This question was then migrated to MathOverflow \cite{339137} and it has become a relatively popular question (with 180 upvotes as of August 2022) on that website. 

It is also worth mentioning that the question of E. Amiot on MathStackExchange was motivated by investigations in the theory of combinatorial tilings and musical composition \cite{tiling:amiot2009}. A tiling of a set $C\subset G$ in an abelian group $(G,+)$ is a pair of sets $A,B$ such that $A\oplus B = C$. In other words, such that every $c\in C$ may be written uniquely as $c=a+b$ where $a\in A$, $b\in B$ and $+$ is the group operation. In case $G=\mathbb Z$ is the additive group of integers, and $C$ is a finite subset, then the problem of finding tilings of $C$ is equivalent to the problem of finding 0-1-polynomial factors of a given 0-1-polynomial. There are of course natural variations on this theme: one could for instance consider a cyclic group $G=\Z/n\Z$, or an abelian group of higher rank $G= \Z\times \Z$. 
\Cref{conj:hybrid} predicts that factorizations of 0-1-polynomials, where both factors are monic with nonnegative coefficients, are always fair. In E. Amiot's words: \emph{(it would) ``explain why Linear Programming in the real fields seems to always find 0-1 factorizations in some tiling problems in dim. 1''} \cite{amiot:mse}. For related works on translational tilings, see e.g. \cite{tiling:kolountzakis2009} and the references in Sloane's OEIS database, Sequence A067824 \cite{oeis}. 

It is natural to ask what happens if in \cref{conj:alg} one replaces polynomials with power series. Equivalently, this amounts to a study of factorization of uniform discrete measures on the natural numbers, but not necessarily with finite support.

About this variation of the main conjecture, we mention a couple of interesting families of counterexamples due to Will Sawin and to Anton Malyshev (in comments to \cite{339137} dated August 25th, 2019, and June 20th, 2020, respectively). In such examples, the power series $\frac 1 {1-x}$, which has all of its coefficients equal to 1, is factored as a product of power series with nonnegative real coefficients. 

The examples of Sawin are given by 
$$
\left(\frac 1 {1-x}\right)^p 
\cdot 
\left(\frac 1 {1-x}\right)^{1-p}
= \frac 1 {1-x},
$$ 
for any $0<p<1$. Note that the Taylor coefficients of $\left (\frac 1 {1-x}\right)^p$ are of the form $(-1)^k\binom{-p}{k}$ for $k\in\N$ and they are easily seen to be nonnegative. This example shows that the unfair 0-1-polynomials conjecture becomes false if we replace polynomials with power series. 

The example of Malyshev essentially amounts to the following identity: 
$$
\left(1+px\right) 
\cdot 
\left(\frac {1 } {1+p} \cdot \frac 1 {1-x} + \frac {p} {1+p} \cdot \frac 1 {1+px} \right)
= \frac 1 {1-x},
$$ 
where $0<p<1$. Note that the Taylor coefficients of the second factor are of the form $\frac {1+p^{k+1}} {p+1}$ for $k\in\N$. 
This example is interesting because one of the factors is a polynomial. 

It would be interesting to to describe more in general the possible factorizations of $\sum_{k=0}^{\infty} x^k$ with nonnegative coefficients. 
These are the same as deconvolutions of the uniform discrete measure on the set of natural numbers $\N$, and perhaps it may be worth to take  inspiration from the already mentioned paper of Lewis \cite{lewis1967}.  
To improve on the previous these examples, it would also be interesting to fabricate a factorizations in which one of the factors is a monic polynomial.

Finally, it is in order to mention some works on factorizations with constrained coefficients.  
Indeed, intuitively, it would make sense that some progress on the conjecture would follow from a study of the divisors and the roots of 0-1-polynomials (Newman polynomials) \cite{poly:new:odlyzko1993,poly:new:smyth1985,poly:new:borw:drungilas2018,poly:new:dubickas2003,poly:new:hare2014,poly:new:mercer2012,poly:new:norm1:campbell1983,poly:new:saunders2017} and of polynomials with nonnegative coefficients \cite{poly:nonn:barnard1999,poly:nonn:dubickas2007,poly:nonn:evans1991,poly:nonn:michelen2020}. Variations on this theme include the study of polynomials with coefficients that belong to a finite set  \cite{poly:finite:borwein2008} or that are constrained by inequalities \cite{poly:restr:borwein}.

\subsection{Statement of results}\label{sec:intro:results}

The purpose of this paper is to prove the following statement
\begin{theorem}\label{main:thm}
	Let $P,Q\in \R[x]$ such that the product $R(x) := P(x)Q(x)$ has only coefficients in $\{0,1\}$ and $R(0)=1$. 
	If $P(x) = x^5 + a x^2 + 1 $ for some $0<a<1$, then $Q(x)$ has some negative coefficient. 
\end{theorem}

One way to rephrase \cref{main:thm}, from the probabilistic point of view, is: 
\emph{If the sum of two independent finite-faced dice is fair, and one of the dice has only three faces labelled with the numbers 0, 2 and 5, then the two dice are fair as well.}

A few remarks are in order regarding the statement of \cref{main:thm}. First, the hypothesis $R(0)=1$ is merely ornamental, as we may always reduce to this case dividing by suitable powers of the variable $x$. 

Next, we explain why the coefficient $a$ is studied in the range $0<a<1$ instead of $a\geq 0$. 
If $a\in\{0,1\}$ then $P(x)$ would be a 0-1 polynomial. It is easy to see that: if $Q(x)$ has only nonnegative coefficients, and $P(x)$ and $R(x)=P(x)Q(x)$ are 0-1-polynomials, then $Q(x)$ is a 0-1-polynomial as well. 
If $a>1$ and $Q(x)$ has only nonnegative coefficients, then the coefficient of $x^2$ in the expansion of $R(x)$ would be strictly greater than 1, which contradicts the hypothesis of $R(x)$ being a 0-1-polynomial.

The statement of the theorem is also true, if we replace $P(x)$ by a polynomial of degree at most 4. This may be shown using an extension of the combinatorial methods of Krasner and Ranulac. In fact, there is no mention of the use of such methods for non-palindromic polynomials, but for lack of space we cannot include a treatment of this extension here. For this, we refer to a followup note of the author (in preparation) \cite{ghilu:01poly:combo}.

It may also be shown that these combinatorial methods are not sufficient to prove \cref{main:thm}  \cite{ghilu:01poly:combo}. In this sense, the case of a factor of the form $P(x)=x^5+a x^2+1$ may be seen as the first family of cases of the main conjecture, in which the classical methods fail. It is exactly this reason that motivated the author to prove \cref{main:thm}. The main objective is to develop new techniques to make progress towards the solution of the unfair 0-1-polynomials conjecture. It was found that a combination of methods of numerical analysis, computer algebra, and the theory of linear recurrence sequences allows us to reach the desired result in the cases examined. 

All in all, our proof shows that the following more precise statement is true:
\begin{theorem}\label{another:thm}
	Let be given a factorization of power series of the form
	$$
	(x^5+a x^2+ 1) \left(\sum_{k=0}^{\infty} b_k x^k\right) = \sum_{k=0}^{\infty} c_{k} x^k,
	$$ 
	with $0<a<1$ and $c_k \in \{0,1\}$ for all $k$. Let moreover $N=\min\{n\geq 4\colon c_n =0\}$ if defined, or $N=\infty$ if $c_n=1$ for all $n\geq 4$. Then there is some natural index $n$ such that $b_n<0$ and $n\leq  N+8$. 
\end{theorem}

Before we move to explain the main points of the proof, let us mention some additional recent progress in producing evidence for the main conjecture. In discussing the problem on mathoverflow \cite{339137} Sil has verified \cref{conj:hybrid} for all 0-1-polynomials $R(x)$ of degree up to 12, using the software Maple. Later on, ``Max Alekseyev has extended the verification to all degrees up to and including 26, which was further improved up to degree 32 by Peter Mueller using Groebner basis calculation (an approach without numerical issues)''. 

It is worth anticipating that one step of the proof of this paper, consists of ruling out the case of 0-1-polynomials $R(x)$ of degree up to 10000, with a factor of the form $P(x) = x^5+ax^2+1$. This verification was achieved using fast computation of resultants in finite fields, performed via the software PARI/GP \cite{parigp}. 
We express the hope that the algorithms used for the present paper (see \cref{lemma:resultant300}) will help with the task of producing new supporting evidence for the main conjecture, or in the search of a counterexample.

\subsection{Summary of the proof} \label{sec:intro:summary}

Since the proof of \cref{main:thm} is quite long, we now attempt to summarize the main ideas in a few paragraphs. The proof consists of three main steps, with which we deal in \cref{sec:coeff}, \cref{sec:large} and \cref{sec:small} respectively. The final section, \cref{sec:numerical}, contains various analytic and numerical auxiliary estimates that are required throughout the proof. 

To start, we assume by contradiction that $Q(x)$ has only nonnegative coefficients and we set up the problem as follows. We fix three polynomials $P,Q,R\in \R[x]$ such that:
\begin{enumerate}
	\item  $P(x)Q(x) = R(x)$;
	\item  $P(x) = x^5 + a x^2 + 1 $ for some $0<a<1$;
	\item  $Q(x) = \sum_{k=0}^{\deg Q} b_k x^k$ has nonnegative coefficients $b_k \geq 0$;
	\item  $R(x) = \sum_{k=0}^{\deg R} c_k x^k$ is a 0-1-polynomial ($c_k\in\{0,1\}$);
	\item $b_0 = c_0 = 1$. 
\end{enumerate} 
It is technically convenient to extend the sequences $\{b_k\}$ and $\{c_k\}$ to all integers, as follows: 
\begin{itemize}
	\item we set $b_k=0$ when $k$ is negative or $k>\deg Q$;
	\item we set $c_k=0$ when $k$ is negative or $k>\deg R$.
\end{itemize} 

\subsubsection{Computing the first 10000 coefficients of $R(x)$ and $Q(x)$}

In the first step we compute explicitly the coefficients $b_{n}$ and $c_n$ for $n\leq 10000$. As a corollary, we deduce that $\deg R\geq 10000$. 

In \cref{prop:deg5} we show with some cheap reasoning that the first few coefficients must be as displayed in \cref{table:deg5}. 

\begin{center}
	\begin{table}[H]
		\centering
		\begin{tabular}{c|cccccc}
			$n$	& 0 & 1 & 2 & 3 & 4 & 5 \\
			\midrule
			$b_n$ & 1 & 0 & $1-a$ & 0 & $1-a+a^2$ & 0  \\
			$c_n$ & 1 & 0 & 1 & 0 & 1 & 1 \\
		\end{tabular}
		\caption{The first few coefficients of $Q(x)$ and $R(x)$.}
		\label{table:deg5}
	\end{table}
\end{center}

Next, we argue that $c_n=1$ for all $4\leq n\leq 10000$, and that $b_n$ is given by a linear recurrence sequence, at least for $n\leq 10000$. 

To see this, we inspect the following relation between the coefficients:
\begin{equation}\label{relationCoeff1}
	c_n = b_n + a b_{n-2} + b_{n-5},
\end{equation}
for every integer $n\in\Z$, coming from the fact that $R$ is the product of $P$ and $Q$. From this relation we deduce two evident facts, 
	\begin{enumerate}[(i)]
		\item if $c_n=0$, then $b_n = b_{n-2} = b_{n-5}=0$;
		\item if $c_n=1$ then $b_n$ is given by the inhomogeneous linear recurrence law 
		\begin{equation}\label{relationRecurrence}
			b_n = 1-a b_{n-2} - b_{n-5}.
		\end{equation}
	\end{enumerate}
We will often use the statement (i) in the contrapositive:
\begin{itemize}
 \item if any of $b_{n-2}$ or $b_{n-5}$ is nonzero, then $c_n=1$.	
\end{itemize}

Using the linear recurrence \eqref{relationRecurrence}, valid as long as $c_n=1$, we get that the values of $b_n$ are given by polynomials in the parameter $a$, see for instance \cref{table:deg9}.

\begin{center}
	\begin{table}[H]
		\centering
		\begin{tabular}{c|ccccc}
			$n$	& $\dots$ & 6 & 7 & 8 & 9  \\
			\midrule
			$b_n$ & $\dots$ & $1-a+a^2-a^3$ & $\quad a\quad $ & $1-a+a^2-a^3+a^4$ & $\quad a-2a^2\quad $ \\
			$c_n$ & $\dots$ & 1 & 1 & 1 & 1  \\
		\end{tabular}
		\caption{More coefficients of $Q(x)$ and $R(x)$.}
		\label{table:deg9}
	\end{table}
\end{center}

In order to prove that $c_n=1$, we need to check that $b_{n-2}$ and $b_{n-5}$ do not vanish simultaneously. We verify this claim for all $n\leq 10000$ and for all $a\not \in\{0,1\}$, using the theory of resultants and the computer software PARI/GP \cite{parigp}. 

\subsubsection{The case of $a$ not too small} 

The second step of the proof deals with the case $0.005 \leq a <1$. 
The main idea is the following: for $n\leq 10000$ we have that $b_n$ is given by a linear recurrence that makes it asymptotically grow at an oscillating exponential rate. Therefore sooner or later we get $b_n<0$, which is a contradiction, or $b_n>1$, which also easily turns into a contradiction. In fact, since $b_{n}+a b_{n-2}+b_{n-5} = c_n\in \{0,1\}$, we see that the condition $b_{n}>1$ implies that either $b_{n-2}$ or $b_{n-5}$ is negative.

We leave all technical details to \cref{sec:large}. Here we only show an example of this strategy, when $a=0.3$.  

If $a=0.3$, then the first few coefficients of $Q(x)$  are shown in \cref{table:deg5:large}. We may continue to compute the coefficients $b_n$ until one of them exits the range $[0,1]$, The first time this happens is at $n=39$, as displayed in \cref{table:deg39:large}.

\begin{center}
	\begin{table}[H]
		\centering
		\begin{tabular}{c|cccccccccc}
			$n$	& 0 & 1 & 2 & 3 & 4 & 5 & 6 & 7 & 8 & 9   \\
			\midrule
			$b_n$ & 1 & 0 & 0.7 & 0 & 0.79 & 0 & 0.763 & 0.3 & 0.771 & 0.12
		\end{tabular}
		\caption{First coefficients $Q(x)$ when $a=0.3$.}
		\label{table:deg5:large}
	\end{table}
\end{center}

\begin{center}
	\begin{table}[H]
		\centering
		\begin{tabular}{c|ccccc}
			$n$	& $\dots$ & 36 & 37 & 38 & 39  \\
			\midrule
			$ b_n$ & $\dots$ & $\approx 0.132655$ & $\approx 0.578565$ & $\approx 0.803317$ & $\approx -0.040588$ 
		\end{tabular}
		\caption{More coefficients of $Q(x)$ when $a=0.3$.}
		\label{table:deg39:large}
	\end{table}
\end{center}

\subsubsection{The case of $a$ small}

If $a$ is too small, the strategy outlined in the previous paragraph does not work, because the exponential growth of the linear recurrence sequence interpolated by coefficients $b_n$ is too slow. More precisely, if $a$ is somewhat smaller than $0.005$, it may well be that $0\leq b_n\leq 1$ for all $0\leq n\leq 10000$.

For $n>10000$ we cannot guarantee that $c_n=1$ anymore. Therefore we cannot assume that $b_n$ follows the linear recurrence \eqref{relationRecurrence} and thus we cannot argue that it has exponential growth.  

Hence, for $0<a<0.005$ we embrace the possibility that $c_n$ vanishes for some large $n$, and we look for a different approach. We let $N$ be the least index $N> 10000$ such that $c_N=0$. As we know from \eqref{relationCoeff1}, this implies that $b_{N-5}=b_{N-2}=b_{N}=0$. 

Since $a$ is small, various quantities involved in the computation of $b_n$ for $n\leq N$ may be estimated with Taylor expansions. This fact gives us a good control of $b_{N-k}$, even for $k\not \in \{2,5\}$: their  values (hypothetical, since we derive them from the absurd assumption that $N$ exists) may be computed exactly, up to a perturbation term that depends on $a$. We also may estimate the value of $N$ itself, see \cref{rmk:N:approx}. 

For example, in case $a=0.003$ we have that $N\approx 11785$ and the hypothetical values of $b_{N-k}$ for small $k$ are forced to be as displayed in \cref{table:N:003}.

\begin{center}
	\begin{table}[H]
		\centering
		\begin{tabular}{c|ccccccc}
			$n$	& $\dots$ & $N-5$ & $N-4$ &  $N-3$ &  $N-2$ &  $N-1$ &  $N$  \\
			\midrule
			$ b_n$ & $\dots$ & $0$ & $\approx 0.81$ & $\approx 0.81$ & $0$ & $\approx 0.5$ & $0$ 
		\end{tabular}
		\caption{Values of $b_{N-k}$ for small $k$, when $a=0.003$.}
		\label{table:N:003}
	\end{table}
\end{center}

The idea now is to continue the stubborn computation of the values of $b_n$ past the index $n=N$. These values may again be computed up to an error term, even if they do not belong anymore to the linear recurrence that interpolates the values $b_{n}$ for $n<N$. As an example, we show in \cref{table:N:003:8} the deduced values of $b_{N+k}$ for small $k$ if $a=0.003$.

\begin{center}
	\begin{table}[H]
		\centering
		\begin{tabular}{c|cccccccc}
			$n$	 & $N+1$ & $N+2$ &  $N+3$ &  $N+4$ &  $N+5$ &  $N+6$ & $N+7$ & $N+8$  \\
			\midrule
			$ b_n$  & $\approx 0.19$ & $\approx 0.19$ & $\approx 0.99$ & $\approx 0.5$ & $\approx 0.99$ & $\approx 0.81$ &  $\approx 0.81$ &  $\approx -0.0018$  
		\end{tabular}
		\caption{Values of $b_{N+k}$ for small $k$,  when $a=0.003$.}
		\label{table:N:003:8}
	\end{table}
\end{center}

We compute a few values $b_{N+k}$ and, lo and behold, $b_{N+8}$ turns out to be negative. This phenomenon persists for the other values of $a$ in the range $0<a<0.005$: in fact, $b_{N+8}$ is calculated to be equal to a \emph{negative} perturbation of zero: $b_{N+8}\approx -0.31 a$.

\subsection{Notation}
\label{theta:notation}

The Big-Theta notation is used to estimate error terms in numerical computations. An expression of the form $A = B+ \Theta(C)$ means $\abs{A-B}\leq C$. 

If $z$ is a non-zero complex number, we define its \emph{argument} $\operatorname{arg} z$ as the unique real number $0\leq \arg z< 2\pi$ such that the equality $z= \abs z e^{i \cdot \arg z}$ holds. 

\section{The first 10000 coefficients of $R(x)$ and $Q(x)$}
\label{sec:coeff}

We first compute the first six coefficients of $R(x)$ and $Q(x)$. 
\begin{proposition}\label{prop:deg5}
	Let $P,Q,R$ be as in \cref{sec:intro:summary}. 
	Then $$ 
	\begin{aligned}
		Q(x) &= 1 + (1-a)x^2 + (1-a+a^2) x^4 + O(x^6),\\
		R(x) &= 1 + x^2 + x^4+ x^5 + O(x^6).
	\end{aligned}
	$$
\end{proposition}

\begin{proof}
	We first deal with the coefficients of even degree. The coefficients relations \eqref{relationCoeff1} for $n=0,2,4$ tell us that  
	\begin{itemize}
		\item $c_0 = b_0$;
		\item $c_2 = b_2 + a b_0$;
		\item $c_4 = b_4 + a b_2$.
	\end{itemize}
	By hypothesis, we know that $c_0=1$ and $b_0=1$. Therefore, $c_2\geq a > 0$ forces $c_2=1$ and so $b_2= 1-a$. Then $b_2$ is also strictly positive, which implies that $c_4=1$ and so $b_4 = 1 - a b_2 = 1-a+a^2$:
	$$
	\begin{array}{lll}
		c_0=1  & c_2 = 1 & c_4 = 1 \\
		b_0 = 1 & b_2 = 1-a & b_4 = 1-a+a^2
	\end{array} 
	$$
	
	The coefficients relations \eqref{relationCoeff1} for $n=1,3,5$ tell us that  
	\begin{itemize}
		\item $c_1 = b_1$;
		\item $c_3 = b_3 + a b_1$;
		\item $c_5 = b_5 + a b_3 + b_0$.
	\end{itemize}

	We now show that $(c_1,c_3,c_5) = (0,0,1)$. We already know that $b_0 =1$. Therefore the above equation for $c_5$ can only hold if $c_5 = 1 $ and $b_5=b_3=0$. But then, the equations for $c_3$ and $c_1$ become
\begin{itemize}
	\item $c_1 = b_1$;
	\item $c_3 = a b_1$.
\end{itemize}
	Since $0<a<1$, these equations are satisfied only if $c_1=c_3=b_1=0$.
	To sum up, we proved that 
	$$
	\begin{array}{llllll}
		c_0=1 & c_1 = 0 & c_2 = 1 & c_3 = 0 & c_4 = 1 & c_5 = 1 \\
		b_0 = 1 & b_1 = 0 & b_2 = 1-a & b_3 = 0 & b_4 = 1-a+a^2 & b_5 = 0
	\end{array}
	$$
	which is what we wanted.
\end{proof}

\subsection{A linear recurrence sequence of polynomials}
	
	Let $B_{n}(t)$ be a sequence of polynomials in a variable $t$, with integer coefficients, defined recursively as follows:
	\begin{itemize}
		\item $B_0(t) = 1$;
		\item $B_1(t) = 0$;
		\item $B_2(t) = 1-t$;
		\item $B_3(t) = 0$ ;
		\item $B_4(t) = 1-t+t^2$;
		\item $B_n(t) = 1 - t B_{n-2}(t) - B_{n-5}(t)$ for all $n\geq 5$. 
	\end{itemize}

These polynomials are defined in such a way that $b_n = B_n(a)$ for $0\leq n \leq 4$. In general, the values of the $b_n$ may be seen as polynomial expressions in the parameter $0<a<1$. These expressions coincide with $B_n(a)$ for some time:

\begin{lemma}\label{coeff:when:c=1}
	Let $N\in\N$ and assume that $c_n=1$ for all $4\leq n< N$. Then $b_n = B_n(a)$ for all $0\leq n< N$.
\end{lemma}
\begin{proof}
	By construction we have $b_n = B_n(a)$ for $0\leq n \leq 4$. 
	By the coefficients relations and by induction on $n$, we get 
	$$
	\begin{aligned}
		b_n &= c_n - a b_{n-2} - b_{n-5} \\
	& = 1 - a B_{n-2}(a) - B_{n-5}(a) \\
 & = B_n(a),
	\end{aligned}
	$$
	for all $n\leq N$. 
\end{proof}

Preparing for the next section, we compute the degree and the leading coefficient of the polynomials $B_n(t)$.

\begin{lemma}\label{Bn:degree:leading}
	For all $k\in\N$ we have that 
	$$ B_{2k}(t) = (-1)^k t^k + \text{ terms of lower degree};$$
	for all $k\geq 3$ we have that
	$$
	B_{2k+1}(t) = (-1)^{k-1}(k-2) t^{k-2} + \text{ terms of lower degree}.
	$$
\end{lemma}
\begin{proof}
	The proof is by induction on $k$. For $k\leq 3$ the statements above may be checked by direct inspection. For $k\geq 4$ we have, by the definition and by the induction hypothesis:
	$$
	\begin{aligned}
		B_{2k}(t) &= 1 - t B_{2k-2}(t) - B_{2k-5}(t) \\
		& = 1 - t (-1)^{k-1} t^{k-1} - \theta_1(t) - (-1)^{k-4}(k-5) t^{k-5} - \theta_2(t) \\
		& = (-1)^{k} t^{k} + \theta_3(t) ,
	\end{aligned}
	$$
	where $\theta_1,\theta_2,\theta_3$ are some polynomials of degree strictly less than $k$. Analogously, we have:
	$$
	\begin{aligned}
		B_{2k+1}(t) &= 1 - t B_{2k-1}(t) - B_{2k-4}(t) \\
		& = 1 - t (-1)^{k-2} (k-3)t^{k-3} - \theta_4(t) - (-1)^{k-2} t^{k-2} - \theta_5(t) \\
		& = (-1)^{k-1} ((k-3)+1)t^{k-2} + \theta_6(t) ,
	\end{aligned}
	$$
	where $\theta_4,\theta_5,\theta_6$ are some polynomials of degree strictly less than $k-2$. The lemma follows by induction. 
\end{proof}

\begin{remark}\label{Bn:constant}
	For the curious reader, we communicate that the constant coefficient of the polynomial $B_n(t)$ is given as follows:
	$$
	B_{2k}(0) = 1 \quad \text{ and }\quad B_{2k+1}(0) = 0,
	$$
	for every $k\in\N$. The proof is a straightforward induction. 
\end{remark}

\subsection{Nonvanishing resultants}

From the coefficients relations \eqref{relationCoeff1}, we see that $c_n$ is allowed to vanish, only if both $b_{n-2}=0$ and $b_{n-5}=0$. In general, having two one-parameter expressions vanish simultaneously should be regarded as a rare event. 

From the theory of polynomial elimination, to test whether two polynomial expressions $f(a),g(a)$ can vanish simultaneously at some value of $a$, it is sufficient to calculate their \emph{resultant} $\op{Res}(f,g)$, which is a special polynomial combination of the coefficients of $f$ and $g$. 
By Sylvester's formula, the resultant may be defined as a determinant:
$$
\op{Res}(f,g) = 
\det\begin{pmatrix} a_0 & a_1 & \cdots & a_d && \\ & \ddots &\ddots &&\ddots & \\ && a_0 & a_1 & \cdots & a_d \\ b_0 & b_1 & \cdots & b_e && \\ & \ddots & \ddots && \ddots & \\ && b_0 & b_1 & \cdots & b_e 
\end{pmatrix} ,
$$
where $d$ and $e$ are the degrees of $f$ and $g$ respectively, and
\begin{equation}\label{def:f,g}
	\begin{aligned}
		f(t)&=\sum_{k=0}^d a_{d-k}t^k\\
		g(t)&=\sum_{k=0}^e b_{e-k}t^k.
	\end{aligned}
\end{equation}
If the polynomials $f$ and $g$ have coefficients in a ring $R$, then the resultant $r= \op{Res}(f,g)$ is an element of $R$. If $R$ is a subring of a field $K$, then we have the following fundamental property of the resultant: $r= 0$ if and only if the system $f(a)=g(a)=0$ has a solution $a$ in an algebraic closure of $K$. Equivalently, $r = 0$ if and only if $f(t)$ and $g(t)$ have a common factor of positive degree, over $K$.

Note that the polynomial formula that defines the resultant depends on the degree of the polynomials. If two polynomials $f,g\in\Z [t]$ have integer coefficients, we may compare their resultant, with the resultant of their reductions $$f\bmod p,\ g\bmod p\in\Z/p\Z [t]$$ modulo a prime number $p$. If $p$ divides the leading coefficient of either $f$ or $g$, then the polynomials will have a drop in degree, when reduced modulo $p$. Otherwise, if $p$ does not divide the leading coefficients, we have that reduction modulo $p$ preserves the degrees and the formula of the resultant is compatible with the reduction. This argument proves the following fact:

\begin{lemma}\label{resultant:mod}
	Let $f,g\in\Z [t]$ be two polynomials with integer coefficients, and let $\op{Res}(f,g)\in\Z$ their resultant. 
	Let $p$ be a prime number and let $f\bmod p,g\bmod p\in\Z/p\Z [t]$ be obtained from $f$ and $g$ by reduction modulo $p$ of the coefficients. Let $\op{Res}(f\bmod p,g\bmod p)\in\Z/p\Z$ be their resultant. Write $f$ and $g$ as in \eqref{def:f,g} and assume that $p$ does not divide $a_0$ and $b_0$. Then
	$$
	\op{Res}(f\bmod p,g\bmod p) = \op{Res}(f,g) \bmod p.
	$$
\end{lemma}

We are ready to apply the theory of the resultants to our problem. 

\begin{lemma}\label{lemma:resultant300}
	Let $5\leq n\leq 10000$ and let $\tau$ be a real number with $\tau\not \in \{0,1\}$. Then $B_{n-2}(\tau)$ and $B_{n-5}(\tau)$ are not both equal to zero.
\end{lemma}

\begin{proof}
	For $5\leq n \leq 10$ we check the claim case-by-case:
	\begin{itemize}
		\item if $n=5$, then $B_{n-5}(t)=1$ does not vanish anywhere;
		\item if $n=6$, then $B_{n-2}(t)=1-t+t^2$ has no real roots;
		\item if $n=7$, then $B_{n-5}(t)=1-t$ vanishes only at $t=1$;
		\item if $n=8$, then $B_{n-2}(t)=1-t+t^2-t^3= (t^2+1)(1-t)$ has $t=1$ as its only real root;
		\item if $n=9$, then $B_{n-5}(t)=1-t+t^2$ has no real roots;
		\item if $n=10$, then $B_{n-2}(t)=1-t+t^2+t^3-t^4$ has no real roots.		
	\end{itemize}
	Next, we write a computer program to test the nonvanishing of the resultant $R_n = \op{Res}(B_{n-2},B_{n-5})$ for $11\leq n \leq 10000$. Since the coefficients of $B_n(t)$ grow exponentially with $n$, a direct computation of these polynomials is unfeasible. One way to remedy that, is to perform the computation modulo $p$, for a few prime numbers $p$. 
	
	For that, we note that $B_{n}(t)$ has integer coefficients, for each $n$. For every prime number $p$ we consider the polynomials $B_{n}(t) \bmod p \in \Z/p\Z[t]$ and, for each $n\geq 11$, the local resultants
	$$
	R_{n,p} := \op{Res}(B_{n-2}\bmod p,B_{n-5}\bmod p) \in\Z/p\Z.
	$$ 
	By \cref{resultant:mod}, the implication
	$$R_{n,p}\neq 0\Longrightarrow R_n\neq 0$$ 
	holds as long as $p$ does not divide the leading coefficients of $B_{n-2}(t)$ and $B_{n-5}(t)$. 
	By \cref{Bn:degree:leading}, this happens in the following cases:
	\begin{itemize}
		\item $p$ does not divide $k-5$, if $n=2k$ is even; 
		\item $p$ does not divide $k-3$, if $n=2k+1$ is odd.
	\end{itemize}
	In \cref{pseudocode:resultant} we describe the algorithm to test the nonvanishing  of $R_n$ for $11\leq n\leq 10000$. 
	We implemented this algorithm in the software PARI/GP \cite{parigp}, see \cref{paricode} for the implementation. We ran the algorithm on a notebook with \textit{Intel(R) Core(TM) i5-4300U CPU 1.90GHz - 2.49 GHz} processor and \textit{8Gb RAM}. After 10-15 minutes of computer time, having run the computation with primes up to $p=17$, the output shows that indeed $R_n\neq 0$ for each $11\leq n\leq 10000$. By the fundamental property of resultants, this means that $B_{n-2}(t)$ and $B_{n-5}(t)$ do not share any complex common root. 
	\begin{algorithm}
		\caption{Pseudocode to test the claim of \cref{lemma:resultant300} for $11\leq n\leq 10000$.} 
		\begin{algorithmic}[1]
			\label{pseudocode:resultant}
			
			\State $MaxN \gets 10000$
			\State $MaxP \gets 17$
			\State \emph{// $Claim_{n}$ refers to the assertion that $B_{n-2}(t)$ and $B_{n-5}(t)$ do not share any common root. We default this variable to $\op{False}$ and try to prove it is $\op{True}$. }
			\State $Claim_n \gets \op{False}$, for all $11\leq n\leq MaxN$
			\State\emph{// We compute resultants modulo $p$ only if $p$ does not divide the leading coefficients of $B_{n-2}$ and $B_{n-5}$}
			\Procedure{DividesLeadingCoefficients}{$n,p$}
			  \If{$n=2k$ and $k\equiv 5 \bmod p$}
			    \Return {$\operatorname{True}$}
			  \ElsIf{$n=2k+1$ and $k\equiv 3 \bmod p$}
			    \Return {$\operatorname{True}$}
			  \Else { \Return {$\operatorname{False}$}}
			    
			  \EndIf
			\EndProcedure
			\State \emph{// Here follows the main algorithm:}
			\For {$p=2,3,5,7,11,\ldots,MaxP$}
			
			  \State \emph{// Polynomials in this loop are elements of $\Z/p\Z[t]$}
			  \State $(B_0,B_1,B_2,B_3,B_4) =  (1,0,1-t,0, 1-t+t^2,0) \mod p$ 
			
			  \For{$n=5,6,\dots,MaxN$}
%
				\If {$n\geq 11$ and $\op{DividesLeadingCoefficients}(n,p)=\operatorname{False}$}
				  \If {$Claim_n = \operatorname{False}$}
			 	    \State \emph{// Compute the resultant of the two polynomials modulo $p$.}
			 	    \State $R_{n,p} \gets Res(B_{n-2},B_{n-5})$
			 	
			 	    \If{$R_{n,p}  \neq 0 \bmod p$}
			 	      \State $Claim_n \gets \operatorname{True}$
			 	    \EndIf
			 	  \EndIf
			 	\EndIf
			 	
			    \State $B_n \gets 1 - t \cdot B_{n-2}  - B_{n-5}$

			  \EndFor
			\EndFor
			\State \emph{// In the end we get that $Claim_n=\operatorname{True}$ for all $11\leq n \leq 10000$. } 
		\end{algorithmic} 
	
	\end{algorithm}

\begin{algorithm}
	\caption{Implementation in PARI/GP of \protect\cref{pseudocode:resultant}}
	\label{paricode}
	\begin{verbatim}
		\\ usage: call algorithm(10000,17)
		default(parisize,50000000);
		
		needtocompute(n,p,nthclaim) = {
			if(nthclaim==1, return(0));
			if(n<11,return(0));
			if(n%2==0,if(((n-10)/2)%p==0,return(0)));
			if(n%2==1,if(((n-7)/2)%p==0,return(0)));
			return(1)}
		
		isresultantnonzero(b,n,p)= {
			r=polresultant(b[(n-2)%6+1],b[(n-5)%6+1]);
			if(r==Mod(0,p),return(0),return(1))}
		
		recurr(b,n,p,nthclaim)= {
			if(needtocompute(n,p,nthclaim)==1,nthclaim=isresultantnonzero(b,n,p));
			\\ to save memory, we ovverride b_{n-6} with b_n 
			\\ lists in PARIGP start with the index 1
			b[n%6+1]= Mod(1,p) - x*b[(n-2)%6+1]-b[(n-5)%6+1];
			return([b,nthclaim])}
		
		localcomputation(N,p,claim)={
			print("Trying with prime ",p," ...");
			u=Mod(1,p);
			\\ initialize [b_0,b_1,b_2,b_3,b_4,b_5] modulo p
			b = [ u , Mod(0,p) , u-u*x , Mod(0,p) , u - u*x+u*x*x , Mod(0,p)];
			for(n=6,N,[b,claim[n]]=recurr(b,n,p,claim[n]));
			for(n=11,N+1, if( claim[n]==0, print("... proved  up to ",n-1); return(claim) ))}
		
		algorithm(N,maxP)={
			claim = listcreate(N+1);
			for(n=1,N+1,listput(claim,0,n));
			forprime(p=2,maxP,claim=localcomputation(N,p,claim));
			return(claim)}
	\end{verbatim}
\end{algorithm}
%
%
%
%
\end{proof}

\subsection{The first 10000 coefficients}
	
\begin{corollary} \label{lemma:recursion300}
	We have $c_n= 1$ for all $4\leq n\leq 10000$ and $b_n = B_n(a)$ for all $0\leq n\leq 10000$. 
\end{corollary}
\begin{proof}
	The proof is by induction. To start, note that $c_4=1$. 
	
	Now, assume that $c_k=1$ for all $4\leq k\leq n-1$, for some $n$ between 5 and 10000. By \cref{coeff:when:c=1} we have $b_k= B_k(a)$ for all $0\leq k\leq n-1$. In particular, we have
	$$ b_{n-2} = B_{n-2}(a) \aand b_{n-5} = B_{n-5} (a). $$
	By \cref{lemma:resultant300}, since $n\leq 10000$, we have that $b_{n-2}$ and  $b_{n-5}$ are not both equal to zero. 
	
	By by the coefficients relations \eqref{relationCoeff1} we have that $ c_n = b_n + a b_{n-2} + b_{n-5} $. Therefore, we see that $c_n$ is strictly positive. Since $c_n\in\{0,1\}$, we deduce that $c_n=1$ to 1. 
	The conclusion is reached by induction on $n$. 
\end{proof}

\subsection {Explicit formula for the recursion}
Let us define $y_n$ by the homogeneous linear recursion
$$
y_n = - a y_{n-2} - y_{n-5}
$$
with the initial values displayed in \cref{table:y:first}.

\begin{center}
	\begin{table}[H]
		\centering
		\begin{tabular}{c|ccccc}
			$n$	& 0 & 1 & 2 & 3 & 4 \\
			\midrule
			$y_n$ & $-\dfrac{1}{2+a}$ & $1-\dfrac{1}{2+a}$ & $-\dfrac{1}{2+a}$ & $1-\dfrac{1}{2+a}$ & $-\dfrac{1}{2+a}$  
		\end{tabular}
		\caption{The initial values of the sequence $y_n$.}
		\label{table:y:first}
	\end{table}
\end{center}

$$
y_0 = - \frac{1}{2+a},
\quad 
y_1 = 1 - \frac{1}{2+a},
\quad y_2 = - \frac{1}{2+a},
\quad y_3 = 1 - \frac{1}{2+a},
\quad y_4 = - \frac{1}{2+a}.
$$

The following lemma relates the above sequence with the sequence of coefficients $b_n$.

\begin{lemma}\label{formula:by}
	Assume that $c_n=1$ for all $4\leq n< N$ for some $N$. Then for all $0\leq n<N$ we have 
	\begin{equation}\label{eq:formula:by}
		b_n = y_{n+3} + \frac 1 {2+a}.
	\end{equation}
\end{lemma}
\begin{proof}
	First, one sees by direct computation that \eqref{eq:formula:by} is true for $0\leq n\leq 4$. Then, we prove it for $5\leq n<N$ by induction on $n$, as follows. 
	By the hypothesis and by \cref{relationRecurrence} we have that the inhomogeneous linear recurrence relation $b_n =1  - a b_{n-2} - b_{n-5}$ holds for all $5\leq n< N$. 
	Then, the induction step is provided by the following computation:
$$
\begin{aligned}
	b_n &= 1  - a b_{n-2} - b_{n-5} \\
	& = 1 - a y_{n+1} - \frac a {2+a} - y_{n-2} - \frac 1 {2+a} \\
	& = y_{n+3} + \frac 1 {2+a}.
\end{aligned}
$$
\end{proof}

Using the theory of linear recurrence sequences, we may calculate $y_n$ explicitly.

\begin{lemma}\label{formula:Yrecurrence}
	For every $n\in\N$ we have that
	\begin{equation}\label{equation:Yrecurrence}
		y_n = \sum_{\rho:\ \pp \rho = 0} c_\rho \rho^n,
	\end{equation}
	where $\rho$ varies among the roots of\footnote{Note that the cubic term is not a misprint: $P(x) $ is defined as $ x^5+ax^2+1$; instead by $\pp x$ we truly mean $x^5 + a x ^3 + 1$.} $\pp x= x^5 + a x ^3 + 1$ and the coefficients $c_\rho$ are given by the formula
	\begin{equation}\label{formula:c}
		c_\rho = -a \frac 1 {2\rho^5 - 3}\frac {\rho^2}{\rho^2 - 1}. 
	\end{equation}
\end{lemma} 

\begin{remark}
	We need to make sure that the formula in \eqref{formula:c} makes sense. In other words, that $(2\rho^5 - 3)$ and $(\rho^2 - 1)$ are nonzero, when $\pp \rho = 0 $. One way to check this is with resultants:
	$$
	\begin{aligned}
		\op{Res} (\pp x,x^2-1) &= -a^2-2a, \\
		\op{Res} (\pp x,2x^5-3) &= -108 a^5-3125.
	\end{aligned}
	$$
	Since none of these resultants vanish if $a>0$, we have that the denominators in \eqref{formula:c} are nonzero.  
	
	The SageMath \cite{sagemath} code to perform this computation is reproduced below
	\begin{center}
		\begin{verbatim}
		R.<a> = PolynomialRing(QQ,'a')
		S.<x> = R[x]
		p = x^5 + a * x^3 + 1
		display(p.resultant(x^2 -1))
		display(p.resultant(2*x^5 - 3))
	\end{verbatim}
	\end{center}
\end{remark}

For the proof of \cref{formula:Yrecurrence}, we require a general result about Vandermonde matrices. The formula stated in \eqref{eq:vandermonde} may be found e.g. in \cite{vandermonde}. For the reader's convenience, we propose a direct derivation for it. The idea is to reduce this formula to a more classical one, which one may find e.g. in \cite{klinger1967vandermonde}.

\begin{lemma}\label{lemma:vandermondeInverse}
	Let $\rho_1,\ldots,\rho_n$ be distinct complex numbers, write
	$$f(x) = x^n + \sigma_1 x^{n-1} + \dots + \sigma_n = (x-\rho_1)\dots (x-\rho_n)$$
	and let $V= V(\rho_1,\ldots,\rho_n)$ be the associated Vandermonde matrix:
	$$ V= (\rho_i^{j-1})_{\substack{1\leq i \leq n\\ 1 \leq j \leq n}}
	:= 
	\begin{pmatrix}
		1& \rho_1& \rho_1^2&  \cdots  &\rho_1^{n-1}\\
		1& \rho_2& \rho_2^2&  \cdots  &\rho_2^{n-1}\\
		 \vdots  &  \vdots  & \vdots  &  \ddots  & \vdots  \\
		1& \rho_n& \rho_n^2&  \cdots  &\rho_n^{n-1}
	\end{pmatrix}.$$ 
	Then the inverse matrix of $V$ is given by:
	\begin{equation}\label{eq:vandermonde}
		 V^{-1} 
		 = \left(  \frac{\rho_j^{n-i} + \sigma_1 \rho_j^{n-i-1} + \dots + \sigma_{n-i}} {f'(\rho_j)} \right)_{\substack{1\leq i \leq n\\ 1 \leq j \leq n}}.
	\end{equation}
\end{lemma}

\begin{proof}
	For each $1\leq j\leq n$, we let $f_j(x):= f(x)/(x-\rho_j)$ and denote its $k$-th coefficient by $\tau_{j,n-k}$, so that
	$$
	f_j(x) 
	= \prod_{\ell\neq j} (x - \rho_{\ell}) 
	= x^{n-1} + \tau_{j,1} x^{n-2} + \dots + \tau_{j,n-1}.
	$$
	We note the following identities:
	\begin{enumerate}[(i)]
		\item $\rho_j^{n-i} + \sigma_1 \rho_j^{n-i-1} + \dots + \sigma_{n-i} = \tau_{j,n-i}$;
		\item $f'(\rho_{j}) = f_j(\rho_{j})$.
	\end{enumerate}
	To prove the identity (i), note that $\sigma_k = \tau_k - \rho_j \tau_{k-1}$ for each $1\leq k\leq n-1$. Then, by substitution, the left-hand side of (i) becomes a telescoping sum with the desired value. 
	
	To prove (ii), note that $f'(x) = \sum_{\ell=1}^n f_\ell(x)$ by Leibniz’ rule for the derivative of a product. Since $f_\ell(\rho_j)=0$ for $\ell\neq j$, we get (ii).
	
	Using the identities above, we see that the scalar product of the $i$-th row of $V$ with the $j$-th column of the matrix in the right-hand side of \eqref{eq:vandermonde} is equal to
	$$
	1\cdot \frac{\tau_{j,n-1}}{f_{j}(\rho_{j})}  
	+ \rho_{i} \cdot \frac{\tau_{j,n-2}}{f_{j}(\rho_{j})} 
	+ \dots 
	+ \rho_{i}^{n-1} \cdot \frac{\tau_{j,0}}{f_{j}(\rho_{j})}
	=  \frac{ f_j(\rho_i)  }{f_{j}(\rho_{j})}  .	
	$$
	This expression is equal to $1$ if $i=j$, and equal to 0 if $i\neq j$. This proves that the matrix displayed to the right of \eqref{eq:vandermonde} is the inverse of $V$. 
%
\end{proof}

\begin{proof}[Proof of \cref{formula:Yrecurrence}]
	First, note that $\pp x$ has five distinct complex roots by \cref{roots:simple}. From the theory of linear recurrence sequences, we get that \eqref{equation:Yrecurrence} for some coefficients $c_\rho$ and all $n\in\N$. These coefficients are uniquely determined by the condition that \eqref{equation:Yrecurrence} holds for $n=0,1,2,3,4$. These five conditions may be encoded as follows. 
	Choose some ordering of the roots of $\pp x$ and let $V$ be the associated Vandermonde matrix. Then let $\underline c$ be the column vector with the five entries $c_\rho$, and $\underline{y}$ be the column vector with entries $y_0,y_1,\ldots,y_4$, which we displayed at the beginning of the section. 
	Then 
	$$ V\cdot \underline c = \underline y. $$
	In other words, $\underline{c} $ is obtained from $\underline y$ multiplying by the inverse of $V$.
	
	Using \cref{lemma:vandermondeInverse}, we get the following formula for $c_\rho$:
	$$ c_\rho = \frac {y_0 (\rho^4 + a \rho ^2 ) + y_1(\rho ^3 + a \rho ) + y_2 (\rho^2 + a) + y_3 \rho + y_4} {5\rho ^4 + 3  a \rho^2}.$$
	
	Since $(y_0,\ldots,y_4) =(0,1,0,1,0)-1/(2+a)$, we have
	$$
	\begin{aligned}
		c_\rho &= \frac 1 {5\rho ^4 + 3 a \rho^2} \left((\rho^3 + a \rho) + \rho - \frac 1 {2+a}(1+\rho+\dots + \rho^4 + a (1 + \rho + \rho^2))\right) \\
		& = \frac 1 {5\rho ^4 + 3 \rho^2} \left( \rho^3 + a \rho + \rho - \frac 1 {2+a}\frac {\rho^5 -1 + a(\rho^3 -1)} { \rho-1} \right).
	\end{aligned}
	$$
	This formula may be simplified as follows.  Using that $\rho^5 + a\rho^3 + 1 = 0 $, we deduce that $\rho^5 -1 + a(\rho^3 -1) = 2+a$ and $5 \rho^4 + 3 a \rho^2 = 2 \rho^4 - 3 \rho^{-1}$. Finally, we get: 
	$$
	\begin{aligned}
		c_\rho &= \frac 1 {2 \rho^4 - 3 \rho^{-1}} \left(-\rho^{-2} + \rho - \frac 1 {2+a}\cdot \frac {2+a} { \rho-1} \right) \\
		&= \frac 1 {2 \rho^5 - 3 } \left(\frac {\rho^4-\rho^3 + \rho^2 - \rho + 1} { \rho (\rho-1)} \right).
	\end{aligned}
	$$
	In order to get the claimed formula, it now suffices to notice that 
	$$ (\rho + 1)(\rho^4-\rho^3 + \rho^2 - \rho + 1) = \rho^5 + 1 = -a \rho^3.$$
\end{proof}

\section{The case of $a$ large: $a\geq 0.005$} \label{sec:large}

In this section we assume that $0.005\leq a < 1$. 
Recall from \cref{theta:notation} that an expression of the form $A = B+ \Theta(C)$ means $\abs{A-B}\leq C$.

\subsection{Estimate the roots $\alpha,\beta,\gamma$ of $\pp x$}\label{sec:abc:abs}

The polynomial $\pp x = x^5 + a x^3 + 1$ has one negative real root $\alpha$ and two pairs of complex conjugate roots which we will denote $\beta, \bar \beta$ and $\gamma, \bar\gamma$. 
We will use the following convention:  $\beta, \bar\beta$ is the complex-conjugate pair with greatest absolute value; $\beta, \gamma$ is the pair of roots with strictly positive imaginary component. 

	If we replace $a$ with a variable real parameter $t$, then for $t\geq 0$ the polynomial $\ppt t x := \pptx t$ has five distinct roots that we denote $\alpha(t),\beta(t),\bar\beta(t) , \gamma(t), \bar\gamma(t)$. We choose them in such a way that they are continuous in the parameter $t$ and moreover $\alpha(a)=\alpha$, $\beta(a)=\beta$ and $\gamma(a)= \gamma$. See \cref{sec:numerical} for details. We also refer to \cref{sec:numerical} for various estimates we need throughout the proof, regarding the functions $\alpha(t)$, $\beta(t)$ and $\gamma(t)$.

We start by recording the following estimates:
\begin{lemma}\label{estimate:abc:abs}
	Suppose that $0.005\leq a< 1$. Then
	\begin{align}
	0.8375\leq &\abs\alpha \leq 0.999002, \\
	1.1872 \geq &\abs \beta \geq  1.000809,\\   
	0.9203 \leq &\abs \gamma \leq 0.999692. 
	\end{align}
\end{lemma}

\begin{proof}
By \cref{roots:increasing} we know that $\abs{\alpha(t)}$ and $\abs{\gamma(t)}$ are decreasing functions of $t\geq 0$, while  $\abs{\beta(t)}$ increases with $t$. Therefore
	$$
	\begin{aligned}
		\abs{\alpha(1)} &\leq \abs\alpha\leq \abs{\alpha(0.005)} \\
		\abs{\beta(1)} &\geq \abs \beta \geq \abs{\beta(0.005)}, \\
		\abs {\gamma(1)} &\leq \abs \gamma \leq \abs {\gamma(0.005)}.
	\end{aligned}
	$$
Now we just need to find and estimate the roots of the polynomial $\ppt {0.005 } x = \pptx {0.005}$, and those of  $\ppt {1 } x = \pptx {1}$. Using any suitable computer program we find the following approximate solutions of $\ppt {0.005} x = 0$:
\begin{align}
	\alpha(0.005)  &= -0.9990010 + \Theta(0.0000001) \\
	\beta(0.005) &= -0.3087072 + 0.9520082 i + \Theta(0.0000001) \label{beta:0.1}\\
	\gamma(0.005) &= 0.8082077 + 0.5883721 i + \Theta(0.0000001).
\end{align} 
So their absolute values are
\begin{align}
	\abs{\alpha(0.005)}  &=  0.9990010 + \Theta(0.0000001) \\
	\abs{\beta(0.005)} &= 1.0008095 + \Theta(0.0000002) \\
	\abs{\gamma(0.005)} &= 0.9996907 + \Theta(0.0000002),
\end{align}
where the Big-Theta notation denotes an estimate of the error term, see \cref{theta:notation}. 
Analogously, the roots of $\ppt 1 x$ are
\begin{align}
	\alpha(1)  &= -0.83762 + \Theta(0.00001) \\
	\beta(1) &= -0.21785 + 1.16695 i + \Theta(0.00001)\\
	\gamma(1) &= 0.63666 + 0.66470 i + \Theta(0.00001),
\end{align} 
and their absolute values are
\begin{align}
	\abs{\alpha(1)}  &=  0.83762 + \Theta(0.00001) \\
	\abs{\beta(1)} &= 1.18711 + \Theta(0.00002) \\
	\abs{\gamma(1)} &= 0.92042 + \Theta(0.00002).
\end{align} 
From these computations, we get the required estimates.
\end{proof}

\begin{lemma}\label{estimate:abc:c}
	Suppose that $0.005\leq a < 1$. Then 
	$$
	\abs{c_\alpha } \leq 2, 
	\quad
	\abs{c_\beta} \geq \frac 1 {2007}, 	
	\quad
	\abs {c_\gamma} \leq 1. 
	$$
\end{lemma}

\begin{proof}
	To estimate $c_\alpha$ we rewrite \eqref{formula:c} using the fact that $-a\alpha^3 = \alpha^5+1$:
	\begin{equation}
		c_\alpha 
		= \frac {\alpha^5+1} {\alpha(2\alpha^5-3)(\alpha^2-1)} 
		= \frac {\alpha^4 -\alpha^3 + \alpha^2 - \alpha +1} {\alpha(2\alpha^5-3)(\alpha-1)}.	
	\end{equation}
Since $\alpha$ is a negative real number that lies in the interval $(-1,0)$, we have the following rough estimates:
\begin{itemize}
	\item $\abs {\alpha^4 -\alpha^3 + \alpha^2 - \alpha +1} 
	\leq 5$,
	\item $\abs {2\alpha^5-3} 
	\geq 3$,
	\item $\abs{\alpha -1} \geq 1$.
\end{itemize}
Moreover, by \cref{estimate:abc:abs} we have $\abs\alpha \geq 0.8375$. Hence, putting all together we get
$$
\abs{c_{\alpha}} \leq \frac{5}{0.8375\cdot 3\cdot 1} \leq 2.
$$
To estimate $\abs {c_\gamma}$ we use the fact that $2\gamma^5-3 = 2(-a\gamma^3 -1)-3  = -2a\gamma^3-5$:
$$
c_\gamma =  \frac { -a\gamma^2} {(2\gamma^5-3)(\gamma^2-1)} 
= \frac {a\gamma ^2} {(5 + 2a \gamma^{3} )(\gamma^2-1)}.
$$

Since $\abs\gamma \leq 1$ and $\abs a \leq 1$, we have the following rough estimates:
\begin{itemize}
	\item $\abs{\gamma^2} \leq 1$,
	\item $\abs{5+ 2a \gamma^3} \geq 3$.
\end{itemize}
Moreover, by \cref{roots:gamma:fifth} we have $\abs{\gamma-1}\geq 0.2$. We claim that this is enough to imply that $\abs{\gamma^2-1} \geq 1/3$. Indeed, write $\gamma = 1-x$, so 
$$\abs{\gamma^2-1} = \abs{-2x+x^2} \geq 2\abs x - \abs x ^2,$$
by the triangular inequality. Solving a quadratic inequality, we see that $2\abs x - \abs x ^2 \geq 1/3$ holds if and only if
$$ 1-\sqrt{2/3} \leq \abs x \leq 1 + \sqrt {2/3}.$$
Now, $1-\sqrt{2/3} = 0.183 + \Theta(0.001)$ is less than $0.2$, so the lower bound holds. The upper bound amounts to $\abs{\gamma -1}\leq  1+ \sqrt {2/3} = 1.816+\Theta(0.001)$. This is true even if we replace $\gamma$ by any complex number with positive real part and absolute value less than 1. 
Putting all together, we get
$$
\abs{c_\gamma} \leq \frac {a\cdot 1} {3 \cdot 1/3} \leq 1.
$$
Finally, we estimate 
$\abs {c_\beta}$. Again, we use the fact that $2\beta^5-3 = -2a\beta^3-5$:
$$
c_\beta = -a \frac { \beta^2} {(2\beta^5-3)(\beta^2-1)} 
= \frac {a} {(5 + 2a \beta^{3} )(1 - \beta^{-2})}.
$$
By \cref{estimate:abc:abs} we know that $1.0008\leq\abs\beta \leq 1.1872$. 
We have therefore the following estimates:
\begin{itemize}
	\item $\abs{1-\beta^{-2}} \leq 1 + 1.0008^{-2} \leq 2$,
	\item $\abs{a^{-1}(5 + 2a \beta^{3} )} \leq 5a^{-1} + 2(1.1872)^3 \leq 5a^{-1} + 3.3466$.
\end{itemize}
Since $\abs a \geq 1/200$, we get
$$
\abs {c_{\beta}} 
\geq \frac 1 {(1000+3.3466)\cdot 2} \geq \frac 1 {2007}.
$$ 
\end{proof}

\subsection{Estimates when $n\approx 10000$}

Using the notation introduced in \cref{sec:abc:abs}, the explicit formula for $y_n$ from \cref{formula:Yrecurrence} may now be written as follows. 
\begin{equation}\label{formula}
	y_n = c_\alpha \alpha^n + 2 \op{Re} ( c_\beta \beta^ n + c_\gamma \gamma^n).
\end{equation} 

We are going to estimate the terms that appear in this formula for $n\approx 10000$, and we will derive a contradiction. 
First, we know \cref{formula:by} and \cref{lemma:recursion300} that  $$y_n= b_{n-3}-  \frac 1 {2+a}$$ for $7\leq n\leq 10003$. We have that $0\leq a\leq 1$ and $0\leq b_n\leq 1$ for every $n$, therefore
 \begin{equation}\label{eq:estimate:y}
 	\abs{y_n}\leq 2 / 3
 \end{equation}
 for $7\leq n\leq 10003$. 
 
By the estimates in \cref{estimate:abc:abs}, we calculate that 
\begin{equation}\label{eq:estimate:abc:300}
\begin{aligned}
	\abs\alpha^{10000} &\leq 0.999002^{10000} \leq 0.0000461,
	\\
	\abs \beta^{10000} &\geq 1.000809^{10000}\geq 3250, 
	\\
	\abs \gamma^{10000} &\leq 0.999692^{10000} \leq 0.0460.
\end{aligned}
\end{equation}

Note that the high powers of $\alpha$ and $\gamma$ are close to zero, therefore \eqref{formula} is telling us that $y_n \approx 2 \Re(c_\beta \beta^n)$ for $n$ large. 
By \eqref{eq:estimate:y} we reach the following conclusion:
\begin{lemma}\label{stima:grande:re}
	Let $n\in\{10000,10001\}$, then
	\begin{equation}
		2 \abs{\op{Re} (c_\beta \beta^n)} \leq 0.76. 
	\end{equation}
\end{lemma}

\begin{proof}
	By \eqref{formula} and the triangular inequality, we have
	$$
	2 \abs{\op{Re} (c_\beta \beta^n)} 
	\leq \abs{y_n} + \abs{c_\alpha\alpha^{n}} + 2\abs{c_\gamma \gamma^{n}} 
	$$
	for every $n\in\N$. If $n\geq 10000$ we have $\abs\alpha^n\leq \abs\alpha^{1000}$ and $\abs \gamma^n \leq \abs\gamma^{10000}$, because $\abs\alpha,\abs\gamma\leq 1$ by \cref{estimate:abc:abs}. 
	If moreover $n\leq 10003$,  we have the inequality $\abs{y_n} \leq 2/3$ by \eqref{eq:estimate:y}. By \cref{estimate:abc:c} we have $\abs {c_\alpha} \leq 2$ and $\abs {c_{\gamma}} \leq 1$. 
	Therefore, using \eqref{eq:estimate:abc:300}, we calculate that
	\begin{equation}
		2 \abs{\op{Re} (c_\beta \beta^n)} 
		\leq 0.6667 + 0.0001 + 0.0920 = 0.7588. 
	\end{equation}
\end{proof}

\subsection{Estimate $\op{Im} \beta$ and conclusion}

The fact that $\abs{\op{Re} (c_\beta \beta^n)}$ is small for $n\approx 10000$ does not contradict the fact that $\abs{\beta^n}$ is large for the same values of $n$. 
Indeed, it may well be the case that $c_\beta \beta^n$ is very close to the imaginary axis in the complex plane. 
However, the following lemma shows that this should not happen for both $c_\beta \beta^n$ and $c_\beta \beta^{n+1}$.

\begin{lemma}\label{trigonometry:re}
	Let $z,w\in\C$ with $\abs w \geq 1$. Then 
	$$
	\max\{2 \abs{\op{Re} (z)}, 2 \abs{\op{Re} (zw)} \} \geq \abs {z} \abs{w}^{-1} \abs{\op{Im}w}.
	$$ 	
\end{lemma}

\begin{proof}
	Write 
	$$ z = \abs z e^{i\theta}, \quad \aand\quad  w = \abs\beta e^{i\phi},$$ with $0\leq \theta,\phi<2\pi$. Replacing if necessary $w$ with $-w$, we may assume that $0\leq \phi<\pi$. 
	
	Given any angle $\tau\in\R$, denote by $\norm \tau = \min_{k\in\Z} \abs{\tau -\pi k}$ the smallest distance from a multiple of $\pi$. 
	It is clear (see \cref{img:trig}) that at least one $\tau$ among $\tau\in\{\theta-\pi/2,\theta+\phi -\pi/2\}$ satisfies either 
	$$\norm \tau \geq \frac \phi 2, 
	\quad \text{or} \quad  
	\norm \tau \geq \frac {\pi -\phi} 2.$$
	
	\begin{figure}[h!]
		\centering
		\includegraphics[width=\linewidth]{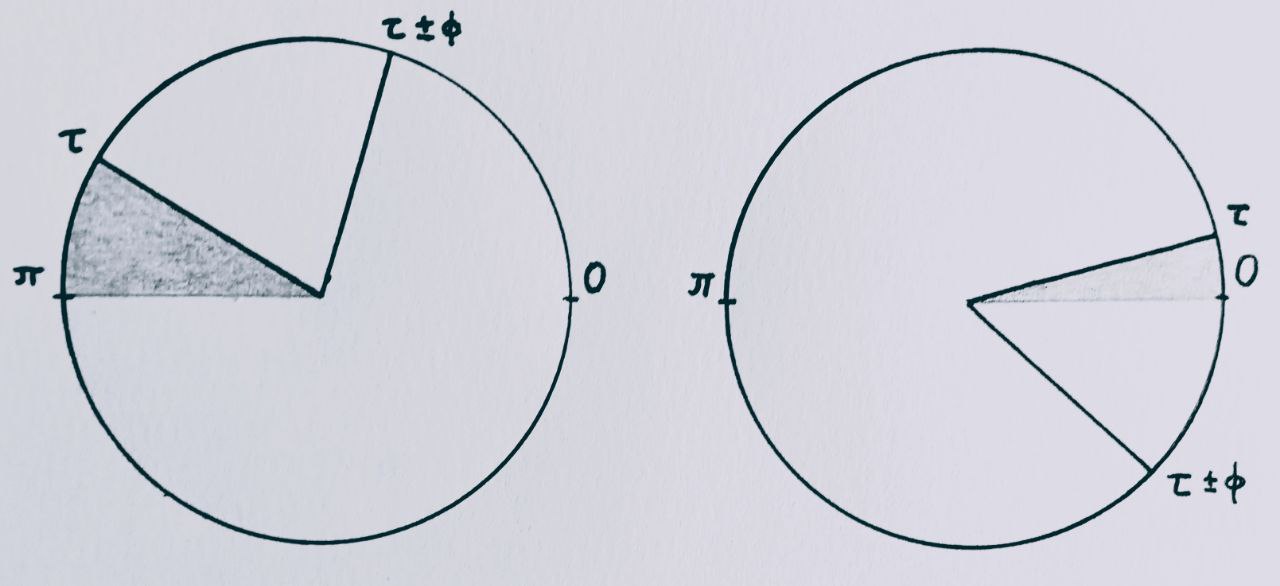}
		\caption{Illustration for \protect\cref{trigonometry:re}. In the picture on the left, we have $\norm\tau\leq (\pi-\phi)/2$. On the right we have $\norm\tau\leq \phi/2$.}
		\label{img:trig}
	\end{figure}
	
	For such $\tau$ we thus either have 
	$$\abs{\sin\tau}\geq \sin (\phi/2)
	\quad \text{or} \quad  
	\abs{\sin\tau}\geq \cos (\phi/2).
	$$
	In both cases, we have
	$$
	\abs{\sin\tau}\geq \sin (\phi/2)\cos (\phi/2) = \frac 1 2  \sin \phi.
	$$
	If $\tau$ is $\theta- \pi/2$, we get
	$$ 
	2 \abs{\Re(z)} = 2 \abs z \cdot \abs {\sin (\theta - \pi/2) }  \geq \abs z \sin \phi.
	$$
	If $\tau$ is $\theta + \phi - \pi/2$ instead, we get
	$$
	2 \abs {\Re (zw)} \geq \abs {zw} \sin \phi.
	$$
	Since $\sin \phi = \abs w^{-1}\Im w $, and $\abs {zw} \geq \abs {z}$, the lemma is proved.
\end{proof}

We now need to estimate the imaginary part of $\beta$.

\begin{lemma}\label{estimate:b:im}
	Let $\beta$ the only root of $\pp x$ that has absolute value larger than 1, and positive imaginary part. Then $\Im (\beta) \geq 0.95$.
\end{lemma}
\begin{proof}
	By \cref{roots:beta:im:increasing} we have that $\Im(\beta(a))$ increases with $a$. Since we are in the case $0.005\leq a<1$, we deduce that 
	$$ \Im\beta \geq \Im(\beta(0.005)).$$
	We have already calculated $\beta(0.005)$ in  \eqref{beta:0.1}: it is equal to	$\beta(0.005) = -0.309 + 0.952 i + \Theta(0.001)$. Then $\Im \beta \geq 0.95$. 
\end{proof}

\begin{corollary}\label{corollary:final:grande}
	We have 
	$$
	\max\{2 \abs{\op{Re} (c_\beta \beta^{10000})}, 2 \abs{\op{Re} (c_\beta \beta^{10001})} \} 
	\geq 1.53.
	$$
\end{corollary}

\begin{proof}
	We have that $\abs{c_\beta}\geq 1/2007$ by \cref{estimate:abc:c}, we have $\abs{\beta^{10000}}\geq 3250$ by \cref{estimate:abc:abs} and $\Im(\beta)\geq 0.95$ by \cref{estimate:b:im}. Then by \cref{trigonometry:re} applied with $z=c_\beta \beta^{10000}$ and $w = \beta$, we get 
	$$
	\max\{2 \abs{\op{Re} (c_\beta \beta^{10000})}, 2 \abs{\op{Re} (c_\beta \beta^{10001})} \}
	\geq \frac{ 3250\cdot 0.95} {2007}  = 1.538 + \Theta(0.001).
	$$
\end{proof}

Note that \cref{corollary:final:grande} contradicts \cref{stima:grande:re}, because $1.53 > 0.76$. As a consequence, \cref{main:thm} is proved for $0.005\leq a<1$. 

%
%
%
%

\section {The case of $ a $ small: $0<a<0.005$} \label{sec:small}

In this section we prove \cref{main:thm} under the assumption that $0<a<0.005$. We refer to \cref{sec:numerical:small} for various technical auxiliary results (mainly quantitative Taylor expansions) used throughout the proofs.

\subsection{Express the various quantities as Taylor expansions}

If the parameter $ a $ is small, the polynomial $ \pp x = x^5 + ax^3 + 1$ may be considered as a perturbation of the polynomial $ x ^ 5 + 1 $.
By the continuity of the dependence of the roots of a polynomial on the coefficients, we therefore deduce that in this case the roots $ \alpha $, $ \beta $, $ \gamma $ can be approximated with the tenth roots of the unity of even order given by $ -1 $, $ e ^ {3 \pi i / 5} $ and $ e ^ {\pi i / 5} $ respectively.

In fact,
 for every root $ \rho $  of $ \pp x $ there exists a tenth root of the unit $ \rho_0 $ of even order, such that we have a Taylor polynomial expansion of the form
$$
\rho = \rho_0 - \frac 1 {5 \rho_0} a + O (a ^ 2),
$$
valid for $ a \to 0 $. For our subsequent analysis, we require precise quantitative estimates. Therefore, it is convenient for us to use the Big-Theta notation introduced in \cref{theta:notation}.

\begin{lemma}\label{taylor:abc}
	Suppose that $0<a<0.005$. Then the following quantitative Taylor expansions hold:
	\begin{align}
		\alpha &= -1 + \frac 1 5 a + \Theta(0.04065a^2);\\
		\beta &= e^{2\pi/5} - \frac {e^{-3\pi i / 5}} 5 a + \Theta(0.04065a^2);\\
		\gamma &= e^{\pi i /5} + \frac {e^{-\pi i /5}} 5 a + \Theta(0.04065a^2) .
	\end{align}
\end{lemma}

\begin{proof}
We prove these estimates in  \cref{Taylor}.
\end{proof}

Using the estimates in \cref{taylor:abc}, it is possible to approximate the coefficients $ c_ \alpha $, $ c_ \beta $ and $ c_ \gamma $ as well. For instance, one may prove by studying the formulae in \cref{formula:c}, that the following zeroth order Taylor expansions hold:  
\begin{align}
	c_\alpha &= -\frac 1 2 + O(a);\\
	c_\beta &= \frac {e^{\pi i / 5}} {10\cos(\pi/10)} a + O(a^2);\\
	c_\gamma &= \frac {e^{-\pi i / 5}} {10\cos(3\pi/10)} a + O(a^2).
\end{align}
See \eqref{c:beta:taylor:qualitative} for an example of this computation. 
However in the proof, we only need simple estimates from above for the absolute values of $c_\alpha$, $c_{\beta}$ and $c_\gamma$.

\begin{lemma}\label{estimate:small:c}
	Suppose that $0< a < 0.005$. Then 
	$$
	\abs{c_\alpha } \leq 0.502, 
	\quad
	\abs{c_\beta} \leq 0.106 \ a, 
	\quad
	\abs {c_\gamma} \leq 0.172\  a. 
	$$
\end{lemma}

\begin{proof}
	Since  $-1<\alpha < -0.999$, we have  $-1< \alpha^5 < 0.995$. Hence: 
	\begin{itemize}
		\item $\abs {\alpha^4 -\alpha^3 + \alpha^2 - \alpha +1} 
		\leq 5$,
		\item $\abs {2\alpha^5-3}  = \abs{-5 + \Theta(2\cdot 0.005)} \geq 4.99$,
		\item $\abs{\alpha -1} \geq 1.999$.
	\end{itemize}
	If we plug these estimates into the formula for $c_\alpha$ we get
	\begin{equation}
		\begin{aligned}
			\abs{c_\alpha} 
			= &\abs{\frac {\alpha^4 -\alpha^3 + \alpha^2 - \alpha +1} {\alpha(2\alpha^5-3)(\alpha-1)}} \\
			\leq  & \frac { 5} {0.999 \cdot 4.99\cdot 1.999}\\
			= & 0.5017 + \Theta(0.0001).
		\end{aligned}	
	\end{equation}

	By \eqref{eq:Taylor:zero} and the inequality $a\leq 1/200$, we have $\gamma(a) = \gamma(0) + \Theta(0.0010045)$. Therefore we have, say, $1-\gamma^2 = 1 - \gamma(0)^2 + \Theta(0.003)$. One may check, using $\gamma(0) = e^{i\pi/5} $, that 
	$$
	\abs{1-\gamma(0)^2} = 1.175  + \Theta(0.001).
	$$ 
	Therefore
	$$
	\begin{aligned}
		\abs {c_{\gamma}} 
		& = \abs {\frac {a\gamma ^2} {(5 + 2a \gamma^{3} )(\gamma^2-1)}}
		\\
		& \leq \frac { a\cdot 1 } { 4.99 \cdot 1.17 }
		\\
		& \leq 0.172 \ a.
	\end{aligned}
	$$
	Finally, to estimate, we use the fact that $\beta = e^{3\pi i /5} + \Theta(0.0010045)$. Therefore we have, say, 
	$$ \abs{1-\beta^2} = \abs{1 - e^{6\pi i /5 }}  + \Theta(0.003) = 1.902 + \Theta(0.004).$$
	Moreover, we have $1\leq \abs\beta \leq 1.000809$, so
	$$
	\begin{aligned}
		\abs {c_\beta} & = \abs {\frac { -a\beta^2} {(5 + 2a\beta^3)(\beta^2-1)} } \\
		&   \leq \frac {a 1.0009^2} { (5 - 0.01\cdot 1.0009^3) \cdot 1.898} \\
		&\leq  0.106 a.
	\end{aligned}
$$
\end{proof}

\subsection{Relative sizes of the roots}

Another estimate we need before the main proof, is a comparison of between the absolute values of the roots of $\pp x$. 
From \cref{taylor:abc} we see that 
$$
\begin{aligned}
	\abs\alpha \approx & 1 - \frac a 5;\\
	\abs\beta \approx & 1 + \frac a 5 \cos(\pi/5); \\
	\abs\gamma \approx & 1 - \frac a 5 \cos (2 \pi /5).
\end{aligned}
$$
Thus, 
$$ 
\abs \alpha \approx  \abs\beta^{-1/\cos(\pi/5)}
\quad \aand\quad 
\abs{\gamma} \approx \abs\beta ^{-\cos(2\pi/5)/\cos(\pi/5)}. 
$$
We make these approximations numerically precise, in the following lemma. 

\begin{lemma}\label{abs:relative}
	If $0<a<0.005$, then
	$$ \abs \alpha =  \abs\beta^{-1.236 + \Theta(0.008)}
	\quad \aand\quad 
	\abs{\gamma} = \abs\beta ^{-0.382 + \Theta (0.005)}
	$$
\end{lemma}
\begin{proof}
	The equations $\abs\alpha = \abs \beta ^{x}$ and $\abs\gamma = \abs \beta ^{y}$ may be solved via the formulae
	\begin{equation}\label{eq:exponents:abs}
		x = \frac {\log \abs \alpha^2} { \log \abs \beta^2}, 
		\quad \quad 
		y = \frac {\log \abs \gamma^2} { \log \abs \beta^2}.
	\end{equation}
	By \cref{taylor:abs2} we have
	$$
	\begin{aligned}
		\abs \alpha ^2& = 1 - \frac {2a} 5 + \Theta(0.13 a^2), \\
		\abs \beta ^2& = 1 + \frac {2a} 5 \cos(\pi/5) + \Theta(0.13 a^2), \\
		\abs \gamma ^2& = 1 - \frac {2a} 5 \cos(2\pi /5) + \Theta(0.13 a^2).
	\end{aligned}
	$$
	In particular, since $a\leq 1/200$, we have 
	$$
	\abs \alpha^2, \abs\beta ^2, \abs\gamma^2  = 1 + \Theta (0.4a 
	 + \frac {0.13}{200} a) = 1 + \Theta(0.41 a).
	$$
	From the quantitative Taylor expansion of the logarithm  (\cref{taylor:log}) we obtain:
	$$
	\begin{aligned}
		\log\abs \alpha ^2& = - \frac {2a} 5 + \Theta(0.13 a^2 + 0.512 \cdot (0.41a)^2) \\
		& =  - \frac {2a} 5 + \Theta(0.22 a^2),\\
		\log \abs \beta ^2& = \frac {2a} 5 \cos(\pi/5) + \Theta(0.22 a^2), \\
		\log \abs \gamma ^2& = - \frac {2a} 5 \cos(2\pi /5) + \Theta(0.22 a^2).
	\end{aligned}
	$$
	The inverse of $\log \abs \beta ^2$ is estimated with \cref{taylor:inv:zero}
	 as follows:
	$$
	\begin{aligned}
		\log \abs \beta ^2& = \frac {2a} 5 \cos(\pi/5) \left( 1 + \Theta\left( \frac {5\cdot 0.22} {2\cdot \cos(\pi /5)} a\right) \right) \\
		& =  \frac {2a} 5 \cos(\pi/5) \left( 1 + \Theta\left( 0.68 a\right) \right), \\
		\frac 1 {\log \abs \beta ^2}& = \frac 5 {2a\cos(\pi/5)}  (1 + \Theta(0.7 a)).
	\end{aligned}
	$$
	Finally, we may estimate the exponents calculated in \eqref{eq:exponents:abs}:
 $$
 \begin{aligned}
 	x &= \frac {-2a/5} { 2a \cos(\pi/5)/5} 
 	\left( 1 + \Theta(\frac {5\cdot 0.22} { 2} a ) \right) 
 	\left( 1+ \Theta(0.7a)\right)\\
 	& = - \frac 1 {\cos (\pi /5) }
 	\left( 1 + \Theta (0.7 a + 0.55 a + 0.385 a^2)\right)\\
 	&= - \frac 1 {\cos (\pi /5)} ( 1+ \Theta (1.26 a)),
 	\\
 	y &= \frac {-2a\cos (2\pi /5)/5} { 2a \cos(\pi/5)/5} 
 	\left( 1 + \Theta(\frac {5\cdot 0.22} { 2\cos(2\pi /5)} a ) \right) 
 	\left( 1+ \Theta(0.7a)\right)
 	\\
 	&  = - \frac {\cos (2\pi /5) } {\cos (\pi /5) }
 	\left( 1 + \Theta (0.7 a + 1.78 a + 1.25 a^2)\right)\\
 	&= - \frac {\cos (2\pi /5) } {\cos (\pi /5) }( 1+ \Theta (2.49 a)).
 \end{aligned}
$$
Now, one may check that  $$
\begin{aligned}
	\frac 1 {\cos(\pi/5) }&= 1.2361 + \Theta(0.0001), \\
	\frac {\cos(2\pi/5)}{ \cos(\pi/5) }&= 0.3820 + \Theta(0.0001).
\end{aligned}
$$ With these numerical values, it is now easy to get the claimed estimates.  
\end{proof}

\subsection{The first $N\geq 4$ with $c_N=0$}
\label{sec:N}

In  \cref{lemma:recursion300} we proved that $c_n=1$ for all $4\leq n\leq 10000$. 
But since $R(x)$ is a polynomial, we must have $c_n=0$ for all $n$ sufficiently large. 

\begin{definition}\label{def:N}
	Let $ N $ be the smallest positive integer $ N \geq 10000 $ for which the $ N $-th coefficient of $ R (x) $ vanishes.
	In other words, suppose that $$ c_N = 0 $$ and that $ c_n = 1 $ for each $ n = 4, \ldots, N-1 $. 
\end{definition}
The vanishing of $ c_N $ implies an exact evaluation of some coefficients of $Q(x)$, as follows.

\begin{lemma}\label{b:y:N:0}
	Let $N$ be as in \cref{def:N}. Then 
	\begin{equation}\label{b:0:N}
		b_{N} = b_ {N-2} = b_ {N-5} = 0 .
	\end{equation}
Moreover, we have
\begin{equation} \label{vanish}
y_ {N + 1} = y_ {N-2} = - \frac 1 {2 + a}.
\end{equation}
\end{lemma}
\begin{proof}
	The triple vanishing stated in \cref{b:0:N} is clear from the formula $b_{N} + a b_ {N-2} + b_ {N-5} = c_N = 0$. The equations in \eqref{vanish} follow from \cref{formula:by}.
\end{proof}

		A glance at the explicit formula for $ y_{N-2} $ and $y_{N+1}$, reveals that their ``dominant terms''  $ 2 \op {Re} (c_ \beta \beta ^ {N-2}) $ and $\Re(c_\beta \beta^{N+1})$  are therefore approximately equal to $-0.5$. One first consequence of this fact is a lower bound on $\abs\beta^{N}$, coming from the estimate $\abs{c_\beta} \leq 0.106 a$ of \cref{estimate:small:c}. 
	
\begin{lemma}\label{bN:large}
	We have $\abs\beta^{N-2} \geq 2.3 \, a^{-1}$.
\end{lemma}

\begin{proof}
	The equations $ b_ {N-2} = b_ {N-5} = 0 $ may be rewritten as follows
	\begin{equation}\label{eq:bN:large:vanish}
		\begin{aligned}
			c_\alpha \alpha^{N-2} + 2 \Re (c_\beta \beta^{N-2})  + 2 \Re(c_\gamma \gamma ^{N-2}) + \frac 1 { 2+a} &= 0 \\
			c_\alpha \alpha^{N+1} + 2 \Re (c_\beta \beta^{N+1})  + 2 \Re(c_\gamma \gamma ^{N+1}) + \frac 1 { 2+a} &= 0 
		\end{aligned}
	\end{equation}
We are going to sum together these equations, but first we need to establish a few estimates. First:
\begin{itemize}
	\item $\abs{\alpha^3+1}\leq 0.6027 a $, because $\abs{\alpha^3+1} =  \abs {1+ \alpha + \alpha^2} \cdot \abs {\alpha -1} \leq 3 \abs {\alpha -1} $ and $\abs{\alpha -1} \leq 0.2009a $ by \eqref{eq:Taylor:zero};
	\item $\abs{\gamma^3+1} \leq 2$, because $\abs\gamma \leq 1$;
	\item $\abs{\beta^3 + 1}\leq 1 + 1.009^3 \leq 2.028$. 
\end{itemize}
Then, by the estimates on $c_\alpha$, $c_{\beta}$, $c_{\gamma}$ of \cref{estimate:small:c}, we get 
\begin{itemize}
	\item $\abs{c_\alpha \alpha^{N-2}(\alpha^3+1)}\leq \abs{c_\alpha} \abs{\alpha^3+1} \leq  0.303 a $;
	\item $\abs{2 \Re(c_\gamma \gamma^{N-2}(\gamma^3+1))} \leq 2\abs{c_\gamma}\abs{\gamma^3+1} \leq 0.688a$;
	\item $\abs{2c_\beta(\beta^3 + 1)}\leq 0.43 a$. 
\end{itemize}
Now we sum side by side the terms of \eqref{eq:bN:large:vanish}. Taking into account the above calculations, and the estimate $2/(2+a) = 1 + \Theta(a/2)$,  we get:
$$
\Theta(0.303 a) + 2 \Re(c_\beta\beta^{N-2}(\beta^3+1)) + \Theta(0.688 a) + 1 + \Theta (0.5 a).
$$
Simplifying the error terms, we get  
$$
2 \Re(c_\beta\beta^{N-2}(\beta^3+1)) = -1 + \Theta (1.491 a).
$$
This implies in particular that 
$$
\abs {c_\beta\beta^{N-2}(\beta^3+1)} \geq 1-\frac 3 2 a.
$$
Since $a\leq 1/200$, we deduce that 
$$
\begin{aligned}
	\abs\beta^{N-2} 
	&\geq \frac {1-3/400} {\abs{c_\beta(\beta^3+1)}} \\
	& \geq \frac {0.9925} {0.43 a} \\
	& \geq 2.3 a^{-1}.	
\end{aligned}
$$
\end{proof}

\begin{corollary}\label{small:ac}
	Let $0<a<0.005$ and let $N$ be as in \cref{def:N}. Then 
	$$
	\abs {c_\alpha \alpha ^{N-2}} \leq 0.054 a,
	\aand
	\abs { 2 c_\gamma \gamma ^{N-2}} \leq 0.035 a.
	$$
\end{corollary}
\begin{proof}
	Recall from \cref{estimate:small:c} that 
	$$
	\abs {c_\alpha} \leq 0.502, 
	\aand 
	\abs {c_\gamma } \leq 0.172 a.
	$$
	Moreover, \cref{bN:large} tells us that $\abs\beta^{N-2} \geq 2.3 a^{-1}$. By \cref{abs:relative} we have that 
	$$
	\abs \alpha \leq \abs \beta ^{-1.228},
	\aand 
	\abs \gamma \leq \abs \beta ^{-0.377}.
	$$
	Therefore 
	$$
	\begin{aligned}
		\abs{c_\alpha \alpha ^{N-2}} 
		& \leq 0.502 \cdot (2.3 a^{-1})^{-1.228} \\
		& \leq 0.502 \cdot \frac {0.005^{0.228} } {2.3^{1.228} } a \\
		& \leq 0.054 a.
	\end{aligned}
	$$
	Analogously,
	$$
	\begin{aligned}
		\abs{2 c_\gamma \gamma ^{N-2}} 
		& \leq 2\cdot  0.172 a \cdot (2.3 a^{-1})^{-0.377} \\
		& \leq 2\cdot 0.172 \cdot \frac {0.005^{0.377} } {2.3^{0.377} } a \\
		& \leq 0.035 a.
	\end{aligned}
	$$
\end{proof}

\subsection{Computing $y_{N+k}$ for small $k$}

The estimates of \cref{sec:N} tell us that $c_\alpha \alpha^n$ and $c_\gamma \gamma^n$ are small in comparison to $c_\beta \beta^n$, when $n\approx N$. Recall that
$y_{n}$ is the real part of $c_\alpha\alpha^n + c_\beta\beta^n + c_\gamma\gamma^n$, by the explicit fomula \eqref{formula}. Then an approximate formula for $y_{N+k}$, when $k$ is small, is $$y_{N+k} \approx \Re(c_\beta\beta^{N+k}).$$ 
We now refine our estimate for $c_\beta\beta^{N-2}$ and then deduce a computation of $c_\beta\beta^{N+k}$ and $y_{N+k}$, for small $k$, up to a small error term.

\begin{lemma}\label{rough:bN}
	\begin{equation}\label{eq:rough:bN}
		 2c_\beta \beta ^{N-2} = \frac {e^{11 \pi i /10}} {2 \cos (\pi/10)} + \Theta (2.26 a) = \frac {e^{11 \pi i /10}} {2 \cos (\pi/10)} (1 + \Theta(4.3 a)) . 
	\end{equation}
\end{lemma}

\begin{proof}
	Let
	\begin{equation}\label{eq:def:ABC}
		A:= c_\alpha \alpha^{N-2}, 
		\quad 
		B:= 2c_\beta \beta ^{N-2}, 
		\aand 
		C:= 2c_\gamma \gamma^{N-2}.
	\end{equation}
Then 
$$y_{N-2+ k} = A\alpha^k + \Re(B\beta^k) + \Re(C\gamma^k)$$
for all $k$. 
Using this formula in particular for $k=0$ and $k=3$, the equations $y_{N-2} = y_{N+1} = -1/(2+a)$ may be written as follows: 
$$
\begin{aligned}
	A+ \Re(B) + \Re(C) = - \frac 1 {2+a}, \\
	A\alpha^3 + \Re(B\beta^3) + \Re(C\gamma^3)= - \frac 1 {2+a}. \\
\end{aligned}
$$
We now use the following estimates:
\begin{itemize}
	\item $1 / (2+a) = 0.5 + \Theta (a/4)$;
	\item $A = \Theta (0.054 a) $ and $C = \Theta (0.035 a)$, by \cref{small:ac};
	\item $\abs \alpha^3, \abs \gamma^3 \leq 1$. 
\end{itemize}
Then we obtain
\begin{equation}\label{system:B}
	\begin{aligned}
		\Re(B)   &= - \frac 1 2 + \Theta(0.339 a), \\
		\Re(B\beta^3) &= - \frac 1 2 + \Theta(0.339 a). \\
	\end{aligned}
\end{equation}
Now, we know that $\beta = e^{3\pi/5} + \Theta (0.2009 a)$, therefore
\begin{equation}\label{eq:beta:3}
	\begin{aligned}
		\beta^3 & = e^{9\pi /5} + \Theta(3\cdot 0.2009 a + 3\cdot 0.2009^2 a^2 + 0.2009^3 a^3)\\
		& = e^{-\pi /5} + \Theta(0.604 a).
	\end{aligned}
\end{equation}
We now let $$ B_0 = - \frac 1 2  - i \frac {\sin (\pi /10)} {2 \cos (\pi/10)}$$
and write $B = B_0 + \delta$. It is clear that $B_0  = e^{11\pi/10}/(2\cos (\pi/10)) $, so all we need to prove is that $\delta = \Theta(2.26a)$. 
Note that $\Re (B_0) = - 1/2$ and 
$$
\begin{aligned}
	\Re(B_0 \beta^3) &= \Re(B_0 e^{-\pi/5}) + \Theta(B_0\cdot 0.604 a)\\
	& = \frac{\Re(e^{11\pi/10}) }{2\cos (\pi/10)} + \Theta \left( \frac {0.604 a} {2\cos (\pi/10)}\right)\\
	& = -\frac 1 2 + \Theta(0.318 a).
\end{aligned}
$$
Subtracting these contributions to the system \eqref{system:B}, we get 
$$
\begin{aligned}
	\Re (\delta) & = \Theta(0.339a),\\
	\Re (\delta\beta^3) & = \Theta (0.657a).
\end{aligned}
$$
If we now apply \cref{trigonometry:re} with $z=\delta$ and $w = \beta^3$, we get
\begin{equation}\label{eq:B:delta}
	2\cdot 0.657a \geq \abs \delta \cdot \abs\beta^{-3} \cdot \abs{\Im(\beta^3)}.
\end{equation}
From \eqref{eq:beta:3} we get the following estimates:
\begin{itemize}
	\item $\Im (\beta^3) = \sin(-\pi/5) + \Theta(0.604 a) = -0.5878 + \Theta(0.0032) $;
	\item $\abs\beta^3 \leq 1 + \Theta(0.0031)$.
\end{itemize}
Hence from \eqref{eq:B:delta} we derive the estimate
$$
\begin{aligned}
\abs \delta &\leq \frac {1.314 a \abs\beta^3 } {\abs \Im(\beta^3)}
\\
&\leq	\frac {1.314 \cdot 1.0031} {0.5846} a\\
& \leq 2.26 a.
\end{aligned}
$$
The second estimate in \eqref{eq:rough:bN} follows from $2.26\cdot 2 \cdot \cos (\pi/10) \leq 4.3$. 
\end{proof}

\begin{corollary}\label{yN}
	For all integer $k$ with $\abs k \leq 11$ we have 
	$$ y_{N+k} = \frac {\cos((3 + 6k) \pi /10} { 2 \cos ( \pi /10) } + \Theta(0.01948).
	$$
\end{corollary}

\begin{proof}
	Using the notation \eqref{eq:def:ABC}, we have the explicit formula  
	\begin{equation}\label{y:ABC}
		y_{N+k}= A\alpha^{k+2} + \Re(B\beta^{k+2}) + \Re(C\gamma^{k+2}). 
	\end{equation}
	First, we are going to use the following estimates:
	\begin{itemize}
		\item $\abs\alpha^{k+2}, \abs\gamma^{k+2} \leq 0.999^{-9}\leq 1.01$;
		\item $\abs A, \abs C \leq 0.054 a \leq 0.00027$ by \cref{small:ac}.
	\end{itemize}
With these estimates we get
$$
y_{N+k}= \Re(B\beta^{k+2})  + \Theta (0.0005454).
$$
Since $\beta = e^{3\pi/5} + \Theta(0.0010045)$, it is not difficult to see that 
$$
\abs{\beta^{k+2} - e^{3(k+2)\pi i/5}} \leq (1.0010045)^{13} - 1 = \Theta(0.014).
$$
By \cref{rough:bN} and $a\leq 1/200$ we have 
$$
B = \frac{e^{11\pi i/10}}{2 \cos(\pi /10)} (1 + \Theta(0.0215) ).
$$
Then 
$$
\begin{aligned}
	B\beta^{k+2} 
	&= \frac{e^{(11+12+ 6k)\pi i/10}}{2 \cos(\pi /10)} (1+ \Theta(0.014) ) (1+ \Theta(0.0215))\\
	& = \frac {e^{(3+6k) \pi i / 10}} {2 \cos(\pi /10)} (1+ \Theta(0.036) ) .
\end{aligned}
$$
Finally, we get that 
$$
\begin{aligned}
	y_{N+k} &= \frac {\Re(e^{(3+6k) \pi i / 10})} {2 \cos(\pi /10)} +   \Theta(\frac {0.036}{2\cos(\pi/10)}) + \Theta(0.0005454)\\
	& = \frac {\cos((3 + 6k) \pi /10)} { 2 \cos ( \pi /10) }  + \Theta(0.01893 + 0.00055).
\end{aligned}
$$
\end{proof}

\begin{center}
	\begin{table}\label{table:y}
		\centering
		\begin{tabular}{cc}
			\toprule
			$n$	& $y_n + \Theta(0.02)$ \\
			\midrule
			$N-3$ & $0$ \\
			$N-2$ & $-0.5$ \\
			$N-1$ & $0.31$ \\
			$N\pm 0$ & $0.31$ \\
			$N+1$ & $-0.5$ \\
			$N+2$ & $0$ \\
			$N+3$ & $0.5$ \\
			$N+4$ & $-0.31$ \\
			$N+5$ & $-0.31$ \\
			$N+6$ & $0.5$ \\
			$N+7$ & $0$ \\
			$N+8$ & $-0.5$ \\
			$N+9$ & $0.31$ \\
			$N+10$ & $0.31$ \\
			$N+11$ & $-0.5$ \\
			\bottomrule
		\end{tabular}
		\caption{Approximate values of $y_{N+k}$ for small $k$.}
	\end{table}
\end{center}

Some values of $y_{N+k}$ for small $k$ are calculated and displayed in \cref{table:y}. To compute these values, it is sufficient to use the estimate provided by \cref{yN}, together with the following numerical approximations:

$$
\begin{aligned}
	\frac {\cos( \pi /10)} { 2 \cos ( \pi /10) } &= 0.5,\\
	\frac {\cos( 3\pi /10)} { 2 \cos ( \pi /10) } &= 0.309017 + \Theta(0.000001),\\
	\frac {\cos( \pi /10)} { 2 \cos ( 5\pi /10) } &= 0.
\end{aligned}
$$

\begin{remark}\label{rmk:N:approx}
	For the curious reader, we show here how it is possible to estimate $N$ itself. Since we do not need this computation in the proof, we are going to show only qualitative estimates. In other words, we use Landau's Big-Oh notation in place of the quantitative more precise Big-Theta notation. 
	
	First, recall that $\beta = e^{3\pi i /5} + O(a)$ and note that $\beta^2-1 = 2 \cos(\pi/10) e^{11\pi i /10}$. Therefore
	\begin{equation}\label{c:beta:taylor:qualitative}
		\begin{aligned}
			c_\beta 
			&= \frac {-a \beta^2}{(5+2a\beta^3)(\beta^2-1)}\\
			& = \frac {-a (e^{6\pi i/5} + O(a))}{(5+O(a))(2 \cos(\pi/10) e^{11\pi i /10}+ O(a))}\\
			& = - \frac{a e^{\pi i / 10}}{10 \cos(\pi/10)} + O(a^2).
		\end{aligned}
	\end{equation}
	Comparing this estimate of $c_\beta$ with the estimate for $c_\beta\beta^{N-2}$ given in \cref{rough:bN}, we get that
	\begin{equation}\label{quantitative:b:N}
		\beta^{N-2} = \frac 5 {2 a} + O(1).
	\end{equation}
	Taking absolute values and logarithms, we get
	$$ (N-2) \log \abs \beta = \log (a^{-1}) + \log (2.5 + O(a)) $$
	Essentially by \cref{taylor:abs2}, we also have that $\abs\beta = 1 + \frac  a 5 \cos(\pi/5) + O(a^2)$, so $\log \abs \beta = \frac  a 5 \cos(\pi/5) + O(a^2)$. 
	Therefore, we obtain:
	\begin{equation}\label{eq:N:approx}
		N =  \frac 5 {\cos(\pi/10)}\cdot a^{-1}\log \left( \frac 5 2 a^{-1}\right) + O(1).
	\end{equation}  
\end{remark}

\begin{remark}\label{rmk:N:approx:10000}
	From \cref{eq:N:approx} we see that $N\geq 10000$ when $a\leq 0.0035$. If $a\approx 0.005$, then $N\approx 6500$.  
\end{remark}

	\subsection{Calculating $b_{N+k}$ for $\abs k$ small}
		
	We now compute the values of $b_n$ for $n\approx N$. Since $c_{N}=0$, the conclusion of \cref{formula:by} is not valid anymore. Therefore, we introduce a new quantity $d_{n}$ to adjust the formula.
		\begin{definition}\label{def:d}
			For every $n\in\N$ let $d_n$ be defined by 
			$$
			b_n = y_{n+3} + \frac 1 {2+a} + d_n.
			$$
		\end{definition}
	
	Note that $d_n=0$ for all $n<N$, by \cref{formula:by}. We now aim to compute $d_{n}$ for some values of $n\geq N$. 
	
	Since $b_n$, $y_n$ and the constant sequence $(1/(2+a))_{n\in\N}$ satisfy linear recurrences with the same companion polynomial, we have that the sequence $d_n$ follows such a linear recursion as well. More precisely, we have
	\begin{equation}\label{recursion:d}
		d_n  = (c_n-1) - a d_{n-2 } - d_{n-5}
	\end{equation}
	for all $n\geq 5$. 
		
	\begin{proposition}\label{d:c:N}
		The values of $c_{N+k}$ and $d_{N+k}$ for $-6\leq k \leq 8$ are those given in \cref{table:d}.
	\end{proposition}

\begin{proof}
	For $-6\leq k\leq -1$ we have that $c_{N+k}=1$ and $d_{N+k}=0$, by the definition of $N$ (\cref{def:N}) and the already mentioned \cref{formula:by}.  For $k=0$ we have $c_N=0$, by definition of $N$, and therefore $d_n=-1$ by \eqref{recursion:d}.
	
	We now claim that $c_{N+k}=1$ for each $1\leq k\leq 8$. For this, by the coefficients relations \eqref{relationCoeff1},  it is sufficient to show that $b_{N+k-2}$ and $b_{N+k-5}$ are not both equal to zero. 
	We know that $b_{N-2}= b_{N-5} = b_{N}=0$, but we claim that $b_{N+k}\neq 0$ for $\abs k \leq 6 $ and $k\not \in\{-5,-2,0\}$. We are about to show part of this claim, by computing values of $b_{N+k}$ for small $k$, up to a relatively  small error. 
	
	Recall, by \cref{yN}, that the values of $y_{N+k+3}$ for $-6\leq k \leq 8$ are those given in \cref{table:d}, up to an error of absolute value at most 0.02. Let us also note that $1/(2+a) = 0.5 + \Theta(a/4)$. Since we do not need such precision, we may simply say that
	$$\frac 1 { 2+a}= 0.5 + \Theta(0.01).$$ 
	We now compute some remaining entries in \cref{table:d}.
	\begin{itemize}
		\item $y_{N-4+3} = 0.31+\Theta(0.02) $ and $d_{N-4}=0$, so $b_{N-4} = 0.81 + \Theta(0.03)$ is nonzero; therefore $c_{N+1}=1$ and so $d_{N+1} = 0 $ by \eqref{recursion:d}.
		\item $y_{N-3+3} = 0.31+\Theta(0.02) $ and $d_{N-3}=0$, so $b_{N-3} = 0.81 + \Theta(0.03)$ is nonzero; therefore $c_{N+2}=1$ and so $d_{N+2} = -a d_{N} = a $ by \eqref{recursion:d}.
		\item $y_{N+1+3} = -0.31+\Theta(0.02) $ and $d_{N+1}=0$, so $b_{N+1} = 0.19 + \Theta(0.03)$ is nonzero; therefore $c_{N+3}=1$ and so $d_{N+3} = 0 $ by \eqref{recursion:d}.
		\item $y_{N-1+3} = 0+\Theta(0.02) $ and $d_{N-1}=0$, so $b_{N-1} = 0.5 + \Theta(0.03)$ is nonzero; therefore $c_{N+4}=1$ and so $d_{N+4} = -a d_{N+2} = -a^2 $ by \eqref{recursion:d}.
		\item $y_{N+3+3} = 0.5+\Theta(0.02) $ and $d_{N+3}=0$, so $b_{N+3} = 1 + \Theta(0.03)$ is nonzero; therefore $c_{N+5}=1$ and so $d_{N+5} = - d_{N} = 1$ by \eqref{recursion:d}.
		\item $y_{N+1+3} = -0.31+\Theta(0.02) $ and $d_{N+1}=0$, so $b_{N+1} = 0.19 + \Theta(0.03)$ is nonzero; therefore $c_{N+6}=1$ and so $d_{N+6} = -a d_{N+4} = a^3$ by \eqref{recursion:d}.
		\item $y_{N+5+3} = -0.5+\Theta(0.02) $ and $d_{N+5}=1$, so $b_{N-4} = -0.5+0.5+1 + \Theta(0.03)$ is nonzero; therefore $c_{N+7}=1$ and so $d_{N+7} = -ad_{N+5} - d_{N+2} = -2a $ by \eqref{recursion:d}.
		\item $y_{N+3+3} = 0.5+\Theta(0.02) $ and $d_{N+3}=0$, so $b_{N+3} = 1 + \Theta(0.03)$ is nonzero; therefore $c_{N+8}=1$ and finally $d_{N+8} = -a d_{N+6} = -a^4 $ by \eqref{recursion:d}.
	\end{itemize}
\end{proof}

\begin{center}
	\begin{table}\label{table:d}
		\centering
		\begin{tabular}{ccccc}
			\toprule
			$n$	& $y_{n+3} + \Theta(0.02)$ & $b_n + \Theta(0.04)$ & $c_n$ & $d_n$\\
			\midrule
			$N-6$ & $0$ & $0.5$ &  1 & 0\\
			$N-5$ & $-0.5$ & $0$ &  1 & 0\\
			$N-4$ & $0.31$ & $0.81$ &  1 & 0\\
			$N-3$ & $0.31$ & $0.81$ &  1 & 0\\
			$N-2$ & $-0.5$ & $0$ &  1 & 0\\
			$N-1$ & $0$ & $0.5$ &  1 & 0\\
			$N\pm 0$ & $0.5$ & $0$ &  0 & $-1$\\
			$N+1$ & $-0.31$ & $0.19$ &  1 & $0$\\
			$N+2$ & $-0.31$ & $0.19$ &  1 & $a$\\
			$N+3$ & $0.5$ & $1$ &  1 & $0$\\
			$N+4$ & $0$ & $0.5$ &  1 & $-a^2$\\
			$N+5$ & $-0.5$ & $1$ &  1 & $1$\\
			$N+6$ & $0.31$ & $0.81$ &  1 & $a^3$\\
			$N+7$ & $0.31$ & $0.81$ &  1 & $-2a$\\
			$N+8$ & $-0.5$ & $0$ &  1 & $-a^4$\\
			\bottomrule
		\end{tabular}
		\caption{Approximate values of $y_{N+k}$, $b_{N+k}$, and values of $c_{N+k}$, $d_{N+k}$, for small $k$.}
	\end{table}
\end{center}

	\subsection{A contradiction: $b_{N+8}$ is negative}
	
	We now have all the ingredients to derive the sought contradiction: $b_{N+8}$ is a negative coefficient of $Q(x)$. We know that $\Re(2c_{\beta}\beta^{N+k+3})$ is the main term in the formula for $y_{N+k+3}$,  and so it is one of the main contributions in the evaluation of $b_{N+k}$. 
	
	Since $\beta$ is approximately equal to a tenth root of unity, we have that $\beta^{10} \approx 1$ and therefore it makes sense to compare the the formulae for $b_{N+8}$ and for $b_{N-2}$, as we do in the following lemma. 
		
	\begin{lemma}\label{root:10}
		Let $\rho$ be a root of $\pp x = x^5 + ax^3+1$. Then 
		$\rho^{10} - 1 = 2 a \rho^3 + a^2 \rho^6$.
	\end{lemma}
\begin{proof}
	Since $\rho^5 =  - a \rho^3 - 1$, we have
	$$
	\begin{aligned}
		\rho^{10} - 1 & = (\rho^5+1)(\rho^5 -1 )\\
		& = -a\rho^3 \cdot (-a\rho^3 -2)\\
		& = 2 a \rho^3 + a^2 \rho^6.
	\end{aligned}
	$$
\end{proof}

	\begin{proposition}\label{b:8:asymp}
		Let $0<a<0.005$ and let $N$ be as in \cref{def:N}. Then  $$b_{N+8} = 2a y_{N+4} + \Theta(a^2).$$ 
	\end{proposition}	

	\begin{proof}
		First, let $A=c_\alpha \alpha^{N-2}$, $B=2c_\beta \beta^{N-2}$ and $C=2c_{\gamma} \gamma^{N-2}$ as in \eqref{eq:def:ABC}. By \cref{def:d} and the explicit formula \eqref{y:ABC}, we have that
		$$
		\begin{aligned}
			b_{N-2} &= A \alpha^{3} + \Re (B\beta^3) + \Re (\gamma^3) + \frac 1 {2+a} + d_{N-2}
			\\
			b_{N+8} &= A \alpha^{13} + \Re (B\beta^{13}) + \Re (C\gamma^{13}) + \frac 1 {2+a} + d_{N+8}	.			
		\end{aligned}
		$$
		 We recall that that $d_{N-2}=0$ and $d_{N+8} = -a^4$ by  \cref{d:c:N}. 
		Hence,
		$$
		b_{N+8} - b_{N-2} = A (\alpha^{13}-\alpha^3)+ \Re (B(\beta^{13}- \beta^3)) + \Re (C(\gamma^{13}- \gamma^3)) - a^4.	
		$$
		Now, we make the following observations:
		\begin{itemize}
			\item $b_{N-2}=0$ by \eqref{b:0:N};
			\item $A = \Theta(0.054a)$ and $C= \Theta(0.035 a)$ by \cref{small:ac};
			\item $\abs \alpha \leq 1$, so by \cref{root:10} we have 
			$$
			\begin{aligned}
				\abs{\alpha^{13}-\alpha^3}&\leq \abs{\alpha^{10} - 1}
				\\
				& \leq 2a + a^2\\
				& \leq 2.005 a;
			\end{aligned}
			$$
			\item analogously, $\gamma^{13}-\gamma^3 = \Theta (2.005 a)$;
			\item $a^4  = \Theta(0.000025 a)$.
		\end{itemize}
		Putting all together, we get
		$$ b_{N+8} =   \Re (B(\beta^{13}- \beta^3)) + \Theta(0.22 a^2).$$
		Now, by \cref{root:10} we have that
		$$
		\Re (B(\beta^{13}- \beta^3))
		 = \Re (2aB\beta^{6}) + \Re (a^2B\beta^{9}).
		$$
		By \cref{rough:bN} and \cref{sharp:abs} we have that 
		$$
		\begin{aligned}
			\abs{B\beta^{9}} &\leq \frac 1 {2 \cos(\pi/10)}\cdot (1+4.3a) \cdot (1.0008097)^9\\
			& \leq 0.5258 \cdot 1.0215 \cdot 1.0074\\
			& \leq 0.55.
		\end{aligned}
		$$
		Therefore, $ b_{N+8} =\Re (2aB\beta^{6}) + \Theta(0.77 a^2)$. From the explicit formula \eqref{y:ABC}, the estimates $\abs\alpha , \abs\gamma\leq 1$, and \cref{small:ac},  we also have that
		$$
		\begin{aligned}
			y_{N+4} &= \Re (B\beta^{6}) + \Theta(A) + \Theta(C)\\
			& = \Re (B\beta^{6}) + \Theta(0.089 a).
		\end{aligned}
		$$
		Finally, comparing $y_{N+4}$ and $b_{N+8}$ we get that 
		$$
		\begin{aligned}
			b_{N+8} &= \Re (2aB\beta^{6}) + \Theta(0.77 a^2)\\
			&  = 2a y_{N+4} + \Theta(a^2(0.77+2\cdot 0.089)),
		\end{aligned}
		$$		
		which proves the claimed estimate. 
	\end{proof}


We now finally calculate that $b_{N+8}$ is a negative coefficient of $Q(x)$. This concludes the proof of \cref{main:thm} in case $0<a<0.005$. 

\begin{corollary}\label{b:8:negative}
	Let $0<a<0.005$ and let $N$ be as in \cref{def:N}. Then $b_{N+8}$ is negative.
\end{corollary}
\begin{proof}
	Recall from \cref{yN} or \cref{table:y}, that $y_{N+4} = -0.31+ \Theta(0.02)$. 
	By \cref{b:8:asymp} we have $b_{N+8} = 2a \cdot (y_{N+4} + \Theta(a))$, therefore
	$$
	\begin{aligned}
		b_{N+8} &\leq a \cdot( 2\cdot (-0.31 + 0.02) + 0.005)\\
		& \leq - 0.5 a.
	\end{aligned}
	$$
	In particular, $b_{N+8}$ is negative.
\end{proof}

\section{Numerical analysis for the roots of $\ppt t x = x^5 + tx^3+1$}\label{sec:numerical}

In this section we study the roots of $\ppt t x= x^5 + tx^3+1$ as the parameter $t$ varies. We derive various numerical estimates that are used throughout the proof of \cref{main:thm} in  \cref{sec:large} and \cref{sec:small}. 

We recall that we denote by $\arg z$ the argument of a non-zero complex number $z$, see \cref{theta:notation}.

\subsection{The five roots $\alpha(t), \beta(t), \bar \beta(t), \gamma(t), \bar\gamma(t)$.}
\label{sec:numerical:generalities}

We start by proving some generalities on the roots of $\ppt t x$ for $t\geq 0$: 
\begin{itemize}
	\item there are five distinct complex roots, one of which is negative real; 
	\item they are located in five distinct angular regions of the complex plane of the form $\{z\colon\arg z\in(2k\pi/5, 2(k+1)\pi/5)\}$, see \cref{img:sectors}; 
	\item as the parameter $t$ varies, the roots of $\ppt tx $ trace five continuous curves in such angular regions. 
\end{itemize}

\begin{figure}[h!]
	\centering
	\includegraphics[width=0.4\linewidth]{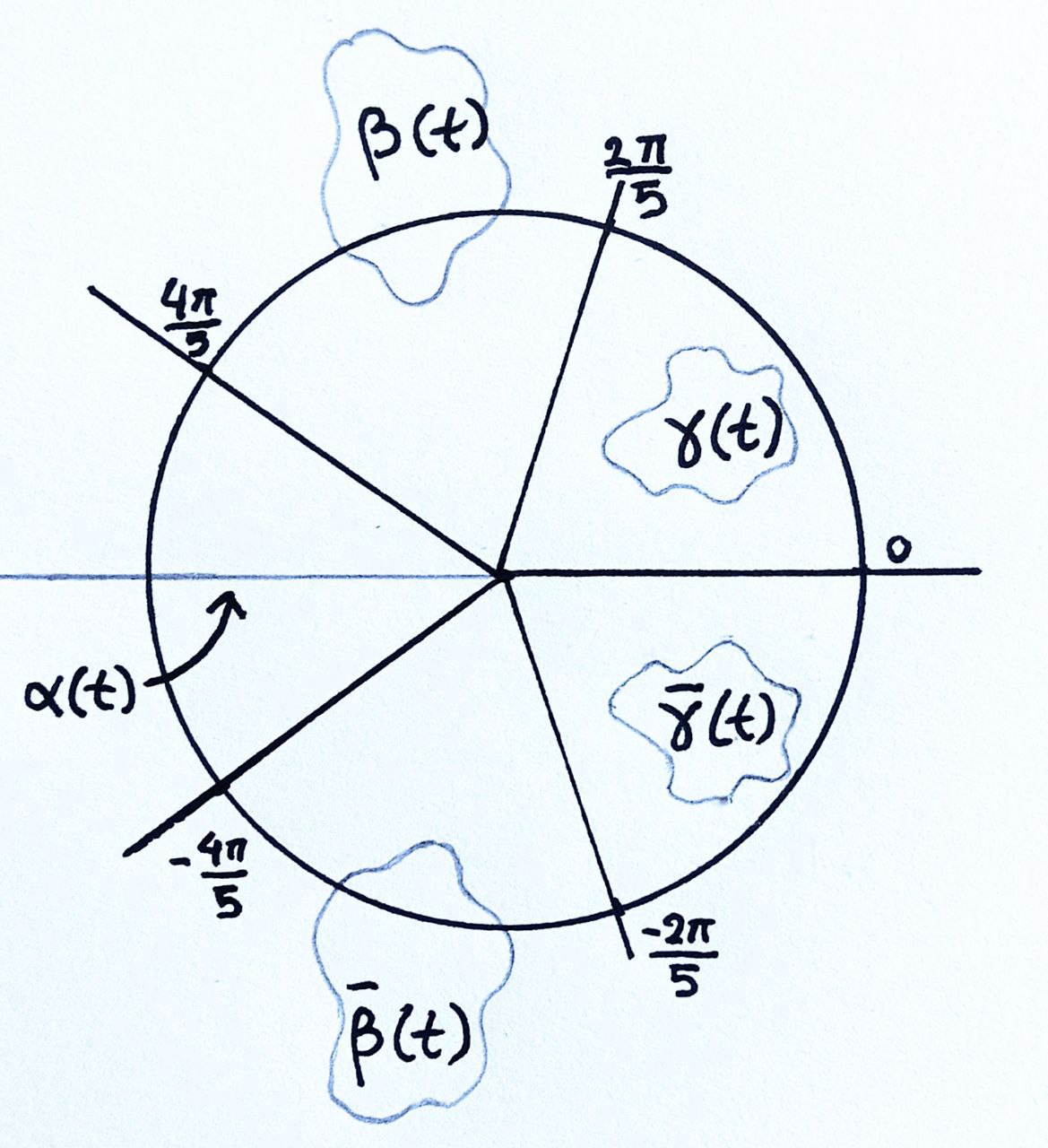}
	\caption{The five roots of $\ppt t x$ and their approximate angular position in the complex plane. }
	\label{img:sectors}
\end{figure}

\begin{lemma}\label{roots:simple}
	If $t\geq 0$, then $\ppt t x$ has five distinct complex roots. 
\end{lemma}

\begin{proof}
	A polynomial has multiple roots if and only if it shares a root with its own derivative, which happens if and only if its discriminant vanishes. The derivative of $\ppt t x $ with respect to $x$ is $\tilde P_t'(x)= 5 x^4 + 3 x^2$. We then calculate the discriminant $\op{Disc}(\pp x) := \op{Res} (\ppt t x,\tilde P_t'(x))$ and get:
	$$
	\op{Disc}(\pp x) = 108 t^5+3125.
	$$
	For $t\geq 0$ this expression does not vanish, therefore by the fundamental property of resultants $\ppt tx$ and $\tilde P_t'(x) $ are coprime. Therefore $\ppt t x$ has only simple roots. 
\end{proof}

\begin{lemma}\label{roots:a:negative}
	If $t\geq 0$, then $\ppt t x $ has exactly one real root, which we denote by $\alpha(t)$. Moreover, such root is negative. 
\end{lemma}
\begin{proof}
	Let  $t\geq 0$ be fixed.  For real $x\in\R$, the assignment $x\mapsto \ppt t x$ defines a continuous and differentiable real function in one parameter. Since $\ppt t 0=1>0$ and $\lim_{x\to -\infty} \ppt t x = -\infty$, we get by continuity that there exists some $\xi\in (-\infty,0)$ such that $\ppt t {\xi} = 0$. 
	
	If $t\geq 0$, the derivative $\tilde P_t'(x)= 5 x^4 + 3t x^2$ is strictly positive for all $x\neq 0$. This means that the real function $x\mapsto \ppt t x$ is strictly increasing on all $\R$. As such, it may vanish at only at $\alpha(t) := \xi$ and nowhere else.  
\end{proof}
%
%

\begin{lemma}\label{roots:arg:forbidden}
	Fix $t\geq 0$ and let $\rho$ be a root of $\ppt t x$. Then $\rho\neq 0 $ and $\op{arg} \rho \not \in \frac {2\pi i} 5 \Z $. 
\end{lemma}

\begin{proof}
	Since $\ppt t 0=1$, we have $\rho \neq 0$. 
	The nonzero complex numbers $z\in\C^\times$ that have argument equal to a multiple of $2\pi i /5$ are exactly those whose fifth power is a positive real number: $z^5\in\R_{>0}$. Suppose that one such number $\rho$ is a root of $\ppt t x$. Then 
	$$-t\rho^3 = \rho^5+1>0. $$
	If $t=0$, this is impossible. If $t>0$ then $\rho^3$ would be a strictly negative real number. But this is also impossible: for instance, it would imply that $\rho^3 = (\rho^3)^6/(\rho^5)^3$ is both a positive and a negative real number.   
\end{proof}

\begin{lemma}\label{roots:continuous}
	For every root $\rho_0$ of $x^5+1$, there exists a unique continuous function $\rho\colon [0,\infty) \to \C$  such that $\rho(t)$ is a root of $\ppt t  x$ for all $t\geq 0$, and $\rho(0)=\rho_0$.
\end{lemma}

\begin{proof}
	This comes from the well-known continuous dependence of the roots of polynomials on their coefficients, as e.g. in \cite{polyconthirose2020}. The argument that follows is a little convoluted but it is elementary and straightforward. More in detail, let $V$ be the complex vector space of monic polynomials with complex coefficients of degree 5, and let $W=\C^5/\!\!\sim$ be the topological space of unordered 5-tuples of complex numbers. Then  the assignment $q(x)\mapsto \text{roots of $q(x)$}$ defines a continuous map $\psi\colon V \to W$. 
	
	Now, $\ppt t x$ for $t\geq 0$ defines a continuous function $[0,\infty)\to V$ and so, by composition with $\psi$, a continuous function $\tilde \rho \colon [0,\infty)\to W$. 
	By \cref{roots:arg:forbidden} each root of $\ppt t x$ lives in the open set $U = \{z\in \C\colon \arg z \not \in (2\pi i / 5)\Z\} $, so the 5-tuple of its roots lives in $U^5/\!\!\sim $. By continuity, the image of $\tilde \rho$ actually sits in the connected component of $U^5/\!\!\sim$ that contains the point $\tilde\rho (0)$. Note that $U$ is the topological disjoint union of 5 connected angular regions of the complex plane $U_k = \{z\in \C\colon \arg z \in (2\pi k i / 5,2\pi (k+1)i / 5 )\}$, for $k= 0,\dots,4$. Note also that each of these regions contains exactly one of the five roots of $x^5+1$. Therefore, we easily see that the connected component of $U^5/\!\!\sim$ which contains $\tilde \rho(0)$ may be identified with $U_0\times U_1 \times \dots U_4$. 
	
	This means that $\tilde \rho$ determines a continuous function $\hat \rho\colon [0,\infty)\to \times_{k=0}^4 U_k$. The five components of $\hat \rho$ are the sought continuous branches $\rho(t)$.
\end{proof}

\begin{corollary}\label{roots:implicit}
	For each root $\rho_0$ of $x^5+1$ the function $\rho(t)$ provided by \cref{roots:continuous} is twice differentiable with continuous second derivative. 
\end{corollary}

\begin{proof}(Sketch of proof)
	In fact, the five continuous roots branches of $\ppt t x$ are real analytic functions of the parameter $t$, hence they are smooth (infinitely differentiable) functions. 
	
	One  way to see this is as follows: by a classical form of the implicit function theorem (see e.g. Thm 4.2.3 in \cite{inversefunctionkrantz2013}) $\rho(t)$ is of class $C^1$. One then sees that $\rho(t)$ is real analytic by the Cauchy-Riemmann equations. An alternative approach consists in deriving explicit formulae for the first and second derivatives of $\rho(t)$, see \eqref{eq:derivatives}.
\end{proof}

We summarize our 

\begin{corollary}\label{roots:arg:range}
	For each $t\geq 0$ the five roots of $\ppt t x$ are given by nonzero complex numbers $\alpha(t), \beta(t), \bar\beta(t), \gamma(t), \bar\gamma(t)$, with 
	\begin{itemize}
		\item $\arg \alpha(a) = \pi$;
		\item $\arg \beta(a)  \in (2\pi/5,4 \pi/5)$;
		\item $\arg \gamma (a) \in (0,2\pi/5)$.
	\end{itemize} 
Moreover the functions $\alpha(t)$, $\beta(t)$, $\gamma(t)$ and their complex conjugates are $C^2$ functions of $t$. 
\end{corollary}

\subsection{Approximate location of the roots for $0\leq t \leq 1$}
\label{sec:numerical:large} 

We may be more precise about the position of the five roots of $\ppt t x$ as $t$ varies. In \cref{img:roots:plot} we display the segment of curves traced by $\alpha(t), \beta(t), \bar\beta(t), \gamma(t), \bar\gamma(t)$ as $t$ varies in the range $0\leq t\leq 1$. 

In the following lemmata we record a few observations that we can get from this picture and that we use throughout \cref{sec:large}:
\begin{itemize}
	\item  the absolute value of $\beta(t)$ increases with $t$, while those of $\alpha(t)$ and $\gamma(t)$ decrease, see \cref{roots:increasing}; 
	\item all roots $\ppt t x$ are bounded away from the fifth roots of unity, see \cref{roots:not:fifth} and \cref{roots:gamma:fifth};
	\item the imaginary part of $\beta(t)$ increases with $t$, see \cref{roots:beta:im:increasing}.
\end{itemize}


\begin{figure}[h!]
	\centering
	\includegraphics[width=0.4\linewidth]{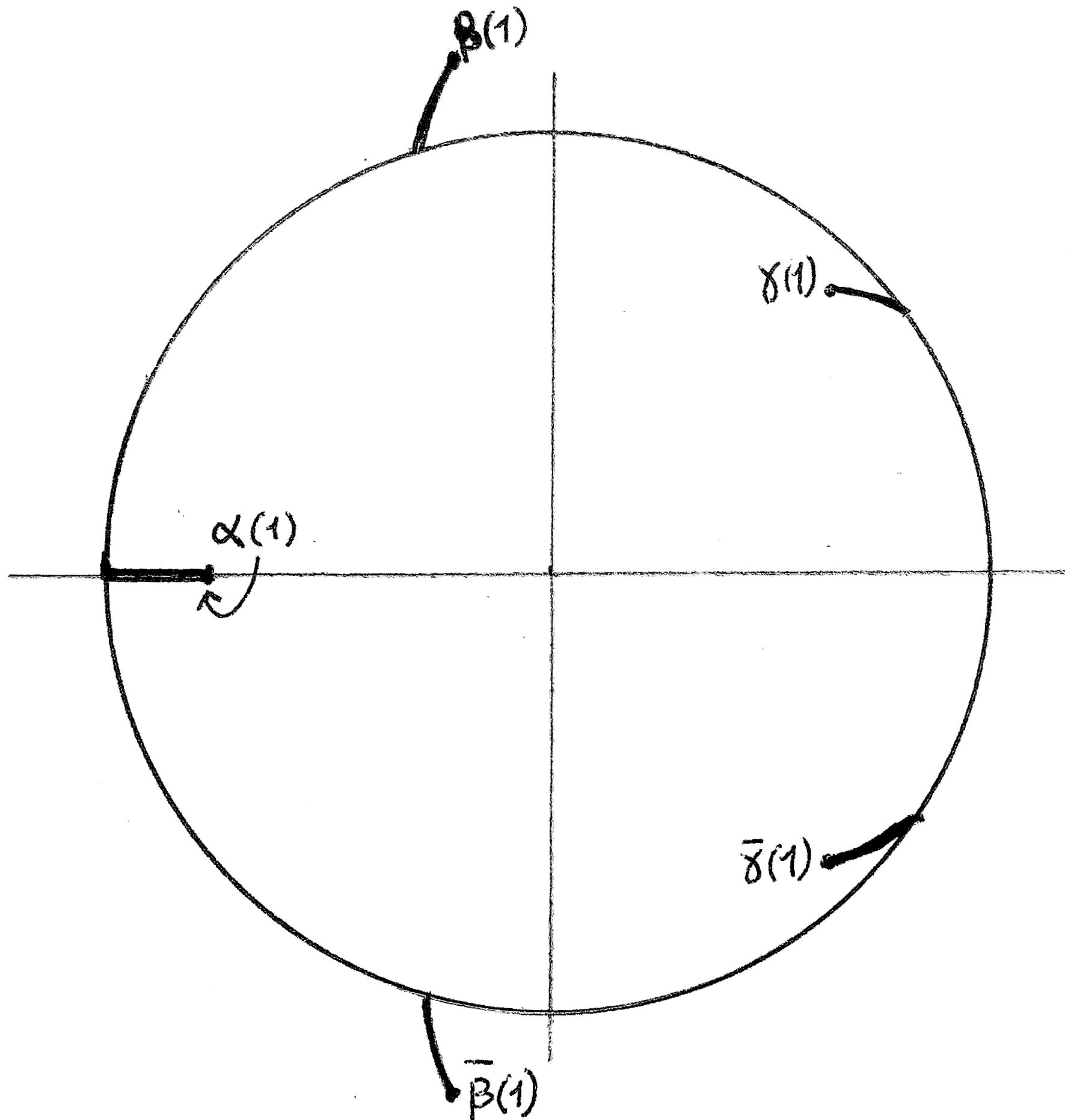}
	\caption{The five roots of $\ppt t x$ for $t=1$ and the arcs traced by them for $0\leq t\leq 1$. }
	\label{img:roots:plot}
\end{figure}

\begin{lemma}\label{gamma:re:decrease}
	Let $t \geq 0$, then
	$$\Re(5 \gamma(t)^2 + 3t) >0.$$
\end{lemma}

\begin{proof}
	If $t=0$, then $\Re(5 \gamma(t)^2 + 3t)  = 5 \cos(2\pi/5)$, which is positive. So, let us assume $t>0$ in the rest of the proof.	By \cref{roots:arg:range} we have that 
	$$ 0 < \arg \gamma(t) < \frac 2 5 \pi .$$
	If $0<\arg \gamma(t) \leq  \pi /4$, then $\Re(\gamma(t)^2) \geq 0$. Since $t$ is positive, we get $\Re(5 \gamma(t)^2 + 3t) >0$ in this case.   
	To deal with the case 
	\begin{equation}\label{eq:53special}
		\pi / 4 <  \arg \gamma(t) < 2\pi / 5,
	\end{equation}
we take a close look at the equation $\gamma^5 + t \gamma ^3 + 1 =0$. Taking imaginary parts, and letting $\theta := \arg \gamma(t)$, we get the following equation:
	\begin{equation}\label{eq:53theta}
			\abs\gamma^5 \sin(5\theta) + \abs \gamma ^3 \sin (3\theta).
	\end{equation}
	Before proceeding further with the proof of \cref{gamma:re:decrease}, let us now show that $\abs \gamma \leq 1$ under the special condition \eqref{eq:53special}  (the inequality $\abs \gamma \leq 1$ is in fact true unconditionally, as we will prove in \cref{roots:increasing}). 
	\begin{lemma}\label{gamma:arg:abs:special}
		If $\pi / 4 <  \arg \gamma(t) < 2\pi / 5$, then $\abs{\gamma(t)}<1$.  
	\end{lemma}
	\begin{proof}
		Assume that $\abs{\gamma(t)} \geq 1$. Then \eqref{eq:53theta} implies that 
		\begin{equation}\label{eq:53absurd}
			\abs {\sin (5\theta)} \leq \abs{\sin(3\theta)}
		\end{equation}  
	Note that $\pi <5 \pi/4 < 5 \theta <  2\pi$, so $\sin (5\theta)$ is negative. Then 
	$$\abs {\sin (5\theta) }
		= \sin(5\theta { \pi})
		= \sin (2\pi - 5\theta) .
	$$
	Moreover $\pi/2 < 3\pi /4 < 3 \theta < 6\pi/5$, so it is not difficult to see that \eqref{eq:53absurd} implies that 
	$$\text{either } \quad 
	3\theta\leq 5\theta -\pi 
	\quad \text{, or }\quad
	3\theta\leq 2\pi - 5\theta.
	$$
	However: the first of these equations is equivalent to $\theta \geq \pi/2$, which is not true; the second is equivalent $\theta \leq \pi/4$, and this also contradicts the assumption we made on $\theta:=\arg\gamma(t)$.
	\end{proof}

	Let us rewrite the target expression for \cref{gamma:re:decrease} in the following fashion:
	$$ 5\gamma^2 + 3\frac {t\gamma^3}{\gamma^3} 
	= 5\gamma^2 - 3\frac {\gamma^5+1}{\gamma ^3} 
= 2\gamma^2 - 3\gamma^{-3}.
$$
Therefore, if we let $\theta:=\arg \gamma(t)$, we have
$$ \Re(5\gamma^2 + 3t)  =  2\abs\gamma^2\cos(2\theta) - 3 \abs\gamma^{-3}\cos(3\theta).$$
Recall that we already settled the case $\theta\leq \pi/4$, so we may assume that $\pi/4 <\theta<2\pi /5$. Moreover, $\abs \gamma \leq 1 $ by \cref{gamma:arg:abs:special}. 

Now, we have $\pi/2 < 2\theta < 4\pi/5$, so $\cos(2\theta)$ is negative and 
$$ 
2\abs\gamma^2\cos(2\theta) 
> 2 \cos (4\pi/5) \approx -1.618\dots $$
On the other hand, we have $3\pi /4<3\theta <6\pi/5$, so $\cos (3\theta)$ is also negative, and
$$
- 3 \abs\gamma^{-3}\cos(3\theta) 
> (-3) \cos (3\pi /4) \approx 2.121\dots
$$ 
Summing the two, we get $\Re(5\gamma^2+3a)> 0.503\dots$, which is positive. 
\end{proof}

\begin{corollary}\label{roots:increasing}
	For $t\geq 0$, the absolute values $\abs {\alpha(t)}$ and $\abs {\gamma(t)} $ are decreasing functions of $t$. In particular, $\abs {\alpha(t))}\leq 1$ and $\abs {\gamma(t)} \leq 1 $. 
	Complementarily, we have that $\abs{\beta(t)}$ increases with $t$, and $\abs{\beta(t)}\geq 1$.  
\end{corollary}

\begin{proof}
	Both $\alpha (t)$ and $\gamma (t)$ are smooth solutions of the equation $x(t)^5 + t x(t)^3 +1 = 0$. Taking derivatives and denoting $\dot x (t):=\partial x(t)/\partial t$, we get
	$5 \dot x (t) x(t)^4 + 3t\dot x(t) x(t)^2 + 3x(t)^3 = 0$. Hence, as long as $5x(t)^4 +3t x(t)^2\neq 0$, we have
	\begin{equation}
		\dot x (t) = -\frac 1 {5x(t) +3t x(t)^{-1}}. 
	\end{equation}
	The square of the absolute value of $x(t)$ is calculated as $\abs {x(t)}^2 = x(t) \bar x (t)$, where $\bar x $ is the complex conjugate of $x$. Then
	\begin{align}
		\frac {\partial \abs{x(t)}^2} {\partial t} 
		&= - \frac {x(t)} {5\bar x(t) +3t \bar x(t)^{-1}} -\frac {\bar x (t)} {5x(t) +3t x(t)^{-1}} \\
		&= - x(t)\bar x(t) \left( \frac 1 {5\bar x(t)^2 +3t } + \frac 1 {5x(t)^2 +3t } \right) \\
		&= -2 \abs{x(t)}^2 \Re\left(\frac 1 {5 x(t)^2 +3t }\right).
	\end{align}
	Now, $\abs{x(t)}$ increases with $t$, if and only if ${\partial \abs{x(t)}^2} / {\partial t}$ is positive. This happens if and only if $\Re\left( 1 /(5 x(t)^2 +3t )\right)$ is negative. In turn, this is equivalent to $\Re(5 x(t)^2 +3t )<0$. 
		
	In case $x(t) = \gamma(t)$, we have that $\Re(5\gamma(t)^2 +3t )>0$ by \cref{gamma:re:decrease}. Therefore $\abs{\gamma(t)}$ is a decreasing function of $t$. 
	
	In case $x(t) = \alpha(t)$, we have that $\alpha(t)$ is a (negative) real number. Therefore $\alpha(t)^2 $ is a nonnegative real number. Since $t\geq 0$, we get $\Re(5\alpha(t)^2 +3t )>0$. Therefore $\abs{\alpha(t)}$ is a decreasing function of $t$. 
	
	By the roots-coefficients relation corresponding to the constant term of the polynomial $x^5+tx^3+1$, we have the following relation between $\alpha, \beta,\gamma$:
	$$ \alpha \beta\bar \beta \gamma \bar \gamma = -1. $$
	Therefore, the product $\abs \alpha \cdot \abs \beta ^2 \cdot \abs \gamma ^2 = 1$ is constant. Since $\abs \alpha$ and $\abs\gamma$ decrease with $t$, the absolute value $\abs \beta$ must decrease with $t$. 

	For $t=0$ we have that $\alpha$, $\beta$, $\gamma$ are roots of unity, therefore $\abs{\alpha(0)}=\abs{\beta(0)} = \abs{\gamma(0)} = 1$. By the preceding monotonicity considerations, we deduce that $\abs{\alpha(t)},\abs{\gamma(t)}\leq 1$, and $\abs{\beta(t)}\geq 1$ for all $t\geq 0$. 
\end{proof}

\begin{lemma}\label{roots:not:fifth}
	Suppose $0\leq t \leq 1$. Let $\rho$ be a root of $\ppt t x$ and let $\mu$ be a fifth root of unity. Then $\abs{\rho-\mu}\geq 1/10$.
\end{lemma}

\begin{proof}
	Suppose on the contrary, that $\rho = \mu + \Theta(1/10)$. Then
	$$
	\begin{aligned}
	\rho^5 &= 1 + 5\Theta \left(\frac 1 {10}\right) +  10\Theta \left(\frac 1 {10^2}\right) +  10\Theta \left(\frac 1 {10^3}\right) + 5\Theta \left(\frac 1 {10^4}\right) +  \Theta \left(\frac 1 {10^5}\right) \\
	&= 1 + \Theta(0.61051).	
	\end{aligned}
	$$
	Therefore 
	$$
	\abs{\rho^5 +1} \geq 2-0.61051 = 1.38949.
	$$
	On the other hand $\rho^5 +1 = -t\rho^3$ and 
	$$
	\abs{\rho^3} \leq 1 + \frac 3 {10} + \frac 3 {100} + \frac 1 {1000} = 1.331.
	$$
	Since $\abs t \leq 1$ we get a contradiction. 
\end{proof}

\begin{lemma}\label{roots:gamma:fifth}
	If $t\geq 0$, we have $\abs{\gamma(t)-1} \geq 0.2$.
\end{lemma}

\begin{proof}
	Suppose on the contrary, that $\gamma(t) = 1 - x$, with $x = \Theta(1/5)$. Then
	$$
	\begin{aligned}
		\gamma(t)^5 &= 1 - 5x + \Theta \left(10 \abs x ^2 + 10 \abs x ^3 + 5 \abs x ^4 + \abs x ^5\right) \\
		&= 1- 5x + \Theta(0.48832)
	\end{aligned}
	$$
	and 
	$$
	\begin{aligned}
		\gamma(t)^3 &= 1 -3x + \Theta \left(3 \abs x ^2 + \abs x ^3\right) \\
		&= 1- 3x + \Theta(0.128).	
	\end{aligned}
	$$
	Therefore
	$$
	\begin{aligned}
		0 &  = 	\gamma(t)^5 + t \gamma^3(t) + 1 \\
		 & = 2+t-x(5+3t) + \Theta(0.48832+0.128 t).
	\end{aligned}
$$
	Since we assumed that $t\geq 0$ and $\abs x \leq 0.2$, we get by the triangular inequality:
	$$
	\begin{aligned}
		\abs {2+t} &\leq \abs x (5+3t) + (0.48832+0.128 t)
\\
	& \leq (0.2\cdot 5 + 0.48832) + (0.2\cdot 3 + 0.128) t\\
	& = 1.48832 + 0.728 t,
	\end{aligned}
	$$
	which is clearly a contradiction.  
\end{proof}

\begin{lemma}\label{roots:beta:im:increasing}
	For $0\leq t\leq 1$  the quantity $\Im(\beta(t))$ increases with $t$.
\end{lemma}

\begin{proof}
	Since $\beta(t)$ is continuous and differentiable in the parameter $t$, it suffices to prove that $\Im \dot \beta(t)>0$ for all $0\leq t \leq 1$. The derivative of $\beta:=\beta(t)$ with respect to the parameter $t$ is
	$$
	\dot \beta  = \frac {-1}{ 5\beta + 3t\beta^{-1}}.
	$$
	Hence, we have
	$$
	\Im \dot \beta >0 \Leftrightarrow \Im (5\beta + 3t \beta^{-1}) >0.
	$$
	Note that $\Im \beta^{-1} = - \Im (\beta / \abs \beta^2)$, so
	$$
	\Im (5\beta + 3a \beta^{-1}) = \Im \beta\cdot  \left(  5 - \frac {3a} { \abs\beta^2}\right). 
	$$
	We know that $\abs{\beta(t)}\geq 1$ for all $t\geq 0$, therefore for $0\leq t \leq 1$ we have 
	$
	5 -  {3t} { \abs\beta^{-2}} \geq 2
	$ is positive. 
	Since $\Im \beta >0$ for all $t$ by construction, we finally deduce  that $\Im \dot \beta >0$. 
\end{proof}

\subsection{Taylor expansions for $0\leq t \leq 0.005$}
\label{sec:numerical:small}

Let $\rho = \rho(t)$ denote any one of the five branches $\alpha(t),\beta(t),\bar\beta(t),\gamma(t),\bar\gamma(t)$ that parameterize the roots of $\ppt t x $. Then $\rho$ is a smooth function of $t$ and its first two derivatives are given by the formulae below:

\begin{equation}\label{eq:derivatives}
	\dot \rho(t) = \frac {-\rho^{-1}} {5 + 3t\rho^{-2}} , 
	\quad \aand \quad 
	\ddot \rho(t) = \frac {2 \rho^{-3}\cdot (5 - 3t \rho^{-2})} { (5+3a \rho^{-2})}.
\end{equation}

\begin{lemma}\label{sharp:abs}
	If $0\leq t \leq 0.005$ we have 
	\begin{equation}\label{eq:sharp:abs}
		0.9990009 \leq \abs{\rho(t)} \leq 1.0008097.
	\end{equation}
\end{lemma}
\begin{proof}
	We have that $\rho(0)=1$ and $\abs{\rho(t)}$ is a monotonic function of $t$ by \cref{roots:increasing}. The roots of $\ppt {0.005} x$ have been computed in the proof of  \cref{estimate:abc:abs}. From these numerical computations, and from the monotonicity, we get that \eqref{sharp:abs} holds 
 	for all $0\leq t \leq 0.005$.
\end{proof}

\begin{corollary}\label{sharp:abs:der}
	If $0\leq t \leq 0.005$, we have 
	$$ \abs {\dot \rho(t)} \leq 0.2009 ,
	\quad \aand  \quad
	\abs{\ddot \rho(t) } \leq 0.0813.
	 $$
\end{corollary}
\begin{proof}
	By \cref{sharp:abs} we have that 
	\begin{equation}
		3 \abs \rho^{-2} \leq 3 \cdot 0.999^{-2} \leq 3.007.
	\end{equation}
	Taking into account that $\abs t \leq 0.005$, we get that $5 + 3t\rho^{-2} = 5 + \Theta (0.0151)$. Hence,
	$$
	\begin{aligned}
		\abs {\dot \rho(t)} &= \frac 1 {\abs \rho \abs{5 + 3t\rho^{-2}}} \\
		&\leq \frac 1 {0.999 \cdot 4.9849}  = 0.2+\Theta(0.00081).
	\end{aligned}
	$$
	Analogously, 
	$$
	\begin{aligned}
		\abs {\ddot \rho(t)} &= \frac {2 \cdot \abs{5 + 3t\rho^{-2}}} {\abs \rho^3 \abs{5 + 3t\rho^{-2}}^3} \\
		&\leq \frac {2\cdot 5.0151} {0.999^3 \cdot 4.9849^3}  = 0.08+\Theta(0.0013).
	\end{aligned}
	$$
\end{proof}

\begin{corollary}\label{Taylor}
	If $0 \leq t \leq 0.005$ we have the first-order Taylor expansion 
	$$
	\rho(t) = \rho(0) - \frac 1 {5 \rho(0) } t + \Theta(0.04065\  t^2).
	$$
	We also have the zeroth-order Taylor expansion
\begin{equation} \label{eq:Taylor:zero}
	\rho(t) = \rho(0) + \Theta (0.2009\ t).
\end{equation}
\end{corollary}

\begin{proof}
	By \eqref{eq:derivatives} we have that $\dot \rho(0) = -1/(5\rho(0))$. Then, by the integral form of the remainder of the Taylor expansion, we have 
	$$ 
	\begin{aligned}
	 	\rho(t) &= \rho(0) + R_0(t) \\
	 	\rho(t) &= \rho(0) - \frac 1 {5 \rho(0) } y + R_1(t)
	\end{aligned}
	$$
	with 
	$$
	\begin{aligned}
		R_0(t) &= \int_0^t \dot \rho(s) ds, \\
		R_1(t) &= \int_0^t \frac {\ddot \rho(s) } 2 (s-t) ds.
	\end{aligned}
	$$
	By the estimates of \cref{sharp:abs:der} we have $\abs{R_0(t)} \leq 0.2009 t$ and $\abs{R_1(t)} \leq 2^{-1}\cdot 0.0813 t^2$, as claimed. 
\end{proof}

\begin{corollary}\label{taylor:abs2}
	If $0\leq t \leq 0.005$, then 
	$$\abs{\rho(t)}^2  = 1 + \frac {2t} 5 \Re(\rho(0))^3 + \Theta(0.13 t^2).$$
\end{corollary}
\begin{proof}
	Recall that $\rho(0)$ is a root of unity such that $\rho(0)^{-1} = -1$. By \cref{Taylor} we have
	$$
	\begin{aligned}
		\abs{\rho(t)}^2  
		&= \rho(t) \widebar{\rho(t)} \\
		& = \rho(0)\widebar{\rho(0)}
	\left(1 + \frac t 5 \rho(0)^3 + \Theta(0.041 t^2)\right)
	\left(1 + \frac t 5 \widebar{\rho(0)}^3 + \Theta(0.041 t^2)\right)\\
	& = 1\cdot \left( 1 + \frac t 5 (\rho(0)^3 + \widebar{\rho(0)}^3) + \Theta\left(0.041t^2(\abs {\rho(t)}+ \abs{\widebar{\rho(t)}}) + \frac {t^2} {25} \right)\right) .
	\end{aligned}	
	$$
	Since $\abs{\widebar{\rho(t)}}=\abs {\rho(t)}\leq 1.0009$, we get that the error term is 
		$$
		\Theta (t^2 \cdot(0.041\cdot 2\cdot 1.0009 + 0.04) = \Theta(0.13 t^2).
		$$
\end{proof}

\begin{lemma}\label{taylor:log}
	Let $x = \Theta (0.01)$, then 
	$$\log (1+x) = x + \Theta (0.512 x^2).$$
\end{lemma}
\begin{proof}
	Let $f(t) = \log (1+t)$ for  $\abs t <1$. Whenever $f(t)$ is defined, we have
	$$ \dot f (t) = \frac 1 {1+t}
	\aand
	\ddot f(t ) =  \frac {-1} {(1+t)^2}.
	$$ 
	By the integral form of the remainder of the Taylor expansion, we have 
	$$
	f(x) = f(0) + x \dot f(0) + R_1(x),	$$
	where 
	$$
	R_1(x) = \int _0 ^x \frac {\ddot f(t)} 2 (x-t) dt.
	$$
	Note that 
	$$ f(0) = 0, \aand \dot f(0) = 1,
	$$
	so the Taylor expansion reads as follows: 
	$$\log (1+x) = x +  R_1(x).$$ 
	Now, we have the following auxiliary estimate:
	\begin{itemize}
		\item $\abs{1/(1+t)^2} \leq (100/99)^2 = \Theta(1.04)$.
	\end{itemize}
	So the error term may be  estimated as follows:
	$$
	\abs {R_1(x) } \leq \frac {x^2} 2 \cdot  1.04 \leq 0.52 x^2.
	$$
\end{proof}

%
\begin{lemma}\label{taylor:inv:zero}
	Let $x = \Theta (0.01)$, then 
	$$\frac 1 {1+x} = 1 +  \Theta (1.02 x).$$
\end{lemma}
\begin{proof}
	The following identity holds:
	$$
	\frac 1 { 1+x}  = 1 - \frac {x} {1+x}.
	$$
	Since $1/(1+x)  =\Theta(100/99) = \Theta(1.02)$, the lemma is proved. 
\end{proof}

\bibliographystyle{abbrv}
\bibliography{biblio_01poly_zot}

\end{document}